\documentclass[reqno,12pt]{amsart}
\usepackage[english]{babel}
\usepackage{amsmath}

\usepackage{amssymb}
\usepackage{graphicx}
\usepackage{epstopdf}

\usepackage{subfigure}
\usepackage{enumitem}

\usepackage{amsaddr}

\newcommand{\bpr}{\begin{trivlist} \item[]{\bf Proof. }}
\newcommand{\epr}{\hspace*{\fill} $\qed$\end{trivlist}}

\newcommand{\be}{\begin{eqnarray}}
\newcommand{\ee}{\end{eqnarray}}
\newcommand{\ba}{\begin{align}}
\newcommand{\ea}{\end{align}}
\newcommand{\bi}{\begin{itemize}}
\newcommand{\ei}{\end{itemize}}

\newcommand{\secref}[1]{Section~\ref{sec:#1}}

\newcommand{\seclab}[1]{\label{sec:#1}}
\newcommand{\eqlab}[1]{\label{eq:#1}}
\renewcommand{\eqref}[1]{(\ref{eq:#1})}

\newcommand{\figref}[1]{Fig.~\ref{fig:#1}}
\newcommand{\figlab}[1]{\label{fig:#1}}
\newcommand{\propref}[1]{Proposition~\ref{proposition:#1}}
\newcommand{\proplab}[1]{\label{proposition:#1}}
\newcommand{\lemmaref}[1]{Lemma~\ref{lemma:#1}}
\newcommand{\lemmalab}[1]{\label{lemma:#1}}

\newcommand{\remref}[1]{Remark~\ref{remark:#1}}
\newcommand{\remlab}[1]{\label{remark:#1}}
\newcommand{\thmref}[1]{Theorem~\ref{theorem:#1}}
\newcommand{\thmlab}[1]{\label{theorem:#1}}

\newcommand{\appref}[1]{Appendix~\ref{app:#1}}

\newcommand{\applab}[1]{\label{app:#1}}

\newtheorem{theorem}{Theorem}[section]
\newtheorem{proposition}[theorem]{Proposition}

\newtheorem{lemma}[theorem]{Lemma}

\newtheorem{remark}[theorem]{Remark}

\numberwithin{equation}{section}
\usepackage{fullpage}
\usepackage{color,comment,xcolor}
\definecolor{orange}{RGB}{255,127,0}

\begin{document}
\title{Revisiting the Kepler problem with linear drag using the blowup method and normal form theory}
\author{K. Uldall Kristiansen}
\address{Department of Applied Mathematics and Computer Science, 
Technical University of Denmark, 
2800 Kgs. Lyngby, 
Denmark }



 \begin{abstract}

In this paper, we revisit the Kepler problem with linear drag. With dissipation, the energy and the angular momentum are both decreasing, but in \cite{margheri2017a} it was shown that the eccentricity vector has a well-defined limit in the case of linear drag. This limiting eccentricity vector defines a conserved quantity, and in the present paper, we prove that the corresponding invariant sets are smooth manifolds. These results rely on normal form theory and a blowup transformation, which reveals that the invariant manifolds are (nonhyperbolic) stable sets of (limiting) periodic orbits. Moreover, we identify a separate invariant manifold which corresponds
 to a zero limiting eccentricity vector. This manifold is obtained as a generalized center manifold over the zero eigenspace of a zero-Hopf point. Finally, we present a detailed blowup analysis, which provides a geometric picture of the dynamics. 
We believe that our approach and results will have general interest in problems with blowup dynamics.
\bigskip
\smallskip

\noindent \textbf{keywords.} Invariant manifolds, nonhyperbolic sets, dynamical systems theory, blowup. 
 \end{abstract}
 \maketitle
\maketitle
\section{Introduction}
In this paper, we consider the Kepler problem with linear drag \cite{margheri2014a,margheri2017a}
\begin{align}
 \ddot u +\delta \dot u + c \frac{u}{\vert u\vert^3}=0,\eqlab{keplerd}
\end{align}
with $u(t)\in \mathbb R^3\backslash \{0\}$, $c>0$ and for $\delta>0$. 
%
%
%
The singularity at $u=0$ corresponds to the collision limit and for $\delta=0$ (no drag/damping), we obtain the classical Kepler problem, whose orbits are conic sections. In fact, as in the classical case, we may scale $u$ and $t$ to achieve $c=1$ so we will assume this henceforth. (In this way, $\delta$ is replaced by $\delta c^{-1/2}$). 
For $\delta=0$, the energy:
\begin{align*}
 E(u,\dot u) = K(\dot u)+P(u),\quad K(\dot u):=\frac12 \vert \dot u\vert^2,\,P(u):= -\frac{1}{\vert u\vert},
\end{align*}
the angular momentum:
\begin{align*}
 L(u,\dot u) = u\wedge \dot u,
\end{align*}
and the eccentricity vector
\begin{align*}
 \mathcal E(u,\dot u) = \dot u \wedge L(u,\dot u)-\frac{u}{\vert u\vert},
\end{align*}
are all conserved quantities. 
For $\delta\ne 0$, we have $\frac{d}{dt}E =-\delta \vert \dot u\vert^2\le 0$ and hence the energy is monotonically decreasing for $\delta>0$. 
Moreover, a simple calculation shows that
\begin{align}
 \frac{d}{dt}L = -\delta L,\eqlab{Lt}
\end{align}
so that 
\begin{align}\eqlab{Lsol}L(u(t),\dot u(t))=e^{-\delta t}L(u(0),\dot u(0)),\end{align} and the angular momentum is therefore exponentially decreasing. Notice, however, that the direction of $L$ is constant and the motion $(u,\dot u)$ is therefore contained in a plane (the orbital plane).  

$\mathcal E$ is also not conserved for $\delta\ne 0$, but in \cite{margheri2017a} it was shown that there exists a limiting eccentricity vector: Let $\phi_t(u,\dot u)$ denote flow associated with \eqref{keplerd}. Then 
\begin{align*}
 \mathcal E_\infty(u,\dot u) = \lim_{t\rightarrow t_{\text{max}}^-} \mathcal E(\phi_t(u,\dot u)),
\end{align*}
exists for all $u,\dot u$. Here $t_{\text{max}}$ is the maximum time of existence (finite for $L=0$, infinite for $L\ne 0$, see \cite{margheri2014a}). In \cite{margheri2017a}, it was shown that the components of $\mathcal E_\infty$ are functional independent and rotationally equivariant: $\mathcal E_\infty (R u,R\dot u)=R \mathcal E_\infty(u,\dot u)$ for all $R\in SO(3)$, and that $\vert \mathcal E_\infty(u,\dot u)\vert$ attains all values in $(0,1]$. If $\vert \mathcal E_\infty(u,\dot u)\vert=1$ then $L=0$, see \cite{margheri2017a}. 
In \cite{margheri2017ab}, the same authors generalized their results to the case, where $\delta$ is a function $\delta(\vert u\vert)$ of $\vert u\vert$ satisfying $\delta(\vert u\vert)\ge c>0$. In \cite{margheri2020a}, the drag $\delta(\vert u\vert)$ was singular at $u=0$, and they showed that in some cases, the limiting eccentricity vector can be discontinuous.

The question of smoothness of $\mathcal E_\infty$ was left open in \cite{margheri2017a}. 
In this paper, we will give a different characterization of $\mathcal E_\infty$, which will allow us to address the issue of smoothness for $\vert \mathcal E_\infty\vert\in (0,1)$. At the same time, using dynamical systems theory, we identify a new smooth invariant manifold corresponding to $\mathcal E_\infty=0$, which acts as a center of the oscillating orbits with $\vert \mathcal E_\infty\vert\in (0,1)$. 
It should be said that this invariant manifold, which will be one-dimensional in a reduced phase space, corresponding to $\mathcal E_\infty=0$, was actually derived as a formal series in \cite{haraux2021a}, which sparked the interest of the present author. Essentially, our proof shows that this series is summable in the sense of Borel-Laplace. 
Separately, using blowup and compactification as our main tools, we provide a geometric description of the  dynamics. This will shed further light on $\vert \mathcal E_\infty\vert\rightarrow 1$. 

The study of dissipation in celestial mechanics has a long history. It even dates back to Jacobi \cite{clebsch2009a}, who introduced dissipative forces of the type $-\delta \vert \dot u\vert^{n-1}\dot  u$; the case $n=1$ corresponds to the linear drag studied in the present paper. There are at least two important mechanisms  for dissipation in celestial mechanics: particle collisions due to nebula (Stokes' dissipation) and solar radiation (Poynting-Robertson dissipation), see \cite{celletti2011a,margheri2014a}.

Corne and Rouche \cite{corne1973a} and  Diacu \cite{diacu1999a} were perhaps the first contributors towards the development of a qualitative theory of the Kepler problem with drag \eqref{keplerd}
for general families of non-constant drag forces $\delta=\delta(u,\dot u)$, depending on $u$ and $\dot u$. \cite{corne1973a} considered \eqref{keplerd} with $\delta (u,\dot u) = k(\vert \dot u\vert)/\vert \dot u\vert$ and showed, under some additional assumptions on the function $k$, that all solutions go to the singularity (potentially in finite time). On the other hand, \cite{diacu1999a} analyzed the qualitative dynamics of the dissipative Kepler problem within a generalized class of Stokes drag; this family includes the important Poynting-Robertson case. 
Many years later, Margheri, Ortega, and Rebelo in \cite{margheri2014a} studied the Kepler problem with linear drag \eqref{keplerd} and provided a more thorough description of the dynamics. In particular, they proved that the system is complete, i.e. solutions exist globally in time, on the set of nonzero angular momentum. Later in \cite{margheri2017a}, the same authors then provided a more geometric description, including the properties of the limiting eccentricity vector. Their results also showed that $\limsup_{t\rightarrow \infty}$ and $\liminf_{t\rightarrow\infty}$ of $\vert u(t)\vert\vert L(t)\vert^{-2}$ both exist along orbits with $\vert \mathcal E_\infty \vert \in (0,1)$,
see also \cite{haraux2021a}. 

In parallel, there has been some studies of dissipative versions of the restricted three body problem, see e.g. \cite{celletti2011a,hadjidemetriou2010a,margheri2012a}. In \cite{celletti2011a} the authors used numerical methods (based upon Fast Lyapunov Indicators (FLI))
to provide information on the different regions of the phase space. They demonstrated both collision and non-collision trajectories. Interestingly, they also documented periodic orbit attractors, but only in the case of linear and Stokes drags. In contrast, in the case of the Poynting-Robertson dissipation, the authors found no other attractors beside the primaries (collisions). 
Following on from this research, \cite{margheri2012a} studied the existence (and nonexistence) of periodic orbits, including Hopf bifurcations around the libration points $L_4$ and $L_5$. 

In this paper, we will use the blowup method and normal form theory to study the dynamics of \eqref{keplerd}. In the context of celestial mechanics and Hamiltonian systems, normal form theory has a long history, dating back to the work of Poincar\'e and Birkhoff. In fact, KAM theory \cite{arnold1989a} itself may be viewed as a normal form theory. 
On the other hand, the blowup method provides a general framework for studying degenerate equilibria in local dynamical systems theory, where the hyperbolic theory (e.g. Hartman-Grobman and center manifolds) does not apply. The rough idea of this approach is to apply a non-invertible transformation (like polar coordinates) that blows up the equilibrium to a sphere (or cylinders in case of lines of degenerate equilibria). 
By appropriately choosing weights associated to the transformation, it is possible (at least for analytic systems) to divide the resulting vector-field by a power of the radius, measuring the distance to the equilibrium. This gives rise to a new vector-field, only equivalent to the original one away from the equilibrium, for which hyperbolicity (or ellipticity) may be (partially) gained on the blowup of the singularity. Sometimes this approach of blowing up equilibria has to be used successively, see e.g. \cite{dumortier2006a}. 

In recent years, this blowup approach has gained prominence within the area of singular perturbation theory, because here degenerate equilibria occur naturally, see \cite{dumortier_1991,dumortier_1993,dumortier_1996,krupa_relaxation_2001}.
In combination, Fenichel's geometric singular perturbation theory and the blowup method has been very successful in describing global phenomena in slow-fast models, see e.g. \cite{kosiuk2011a,kosiuk2015a,kristiansen2017a,2019arXiv190312232U,uldall2021a}. More recently, the blowup approach has been generalized with the purpose of ``gaining smoothness'', rather than hyperbolicity, in the context of smooth systems approaching nonsmooth ones, see \cite{jelbart2021a,jelbart2021b,kristiansen2018a,Llibre07}. 

Obviously, blowup can also take on a different meaning in mathematics, namely (finite time) blowup of solutions of (ordinary or partial) differential equations. 
In dynamical systems theory, blowup solutions can be studied by Poincar\'e or Poincar\'e-Lyapunov compactification (which in fact bear some resemblance to the blowup method), see \cite{dumortier2006a}. Upon compactification one can study equilibria, again after proper desingularization of the vector-field, at infinity and these can be analyzed by local methods of dynamical systems theory. In particular, such points at infinity are sometimes completely degenerate, which can then be resolved by the blowup method. 

Related to blowup of solutions is the existence of solutions approaching true singularities of ordinary differential equations; like the collision limit $u=0$ for \eqref{keplerd} where the associated vector-field is ill-defined. 
The analysis of collision (as well as near-collision) solutions in the $n$-body problem has a long history, also dating back to Poincar\'e, see also \cite{celletti2006a,duignan2020a,duignan2021a}. For the two-body problem, which can be reduced to \eqref{keplerd}$_{\delta=0}$, there is a change of coordinates (the Levi-Cevita transformation) and a nonlinear transformation of time (related to desingularization), that transforms the collision into a regular point of the equations. The Levi-Cevita transformation is -- similar to a blowup transformation -- not invertible; it is in fact a double cover. More generally in the $n$-body problem, the Levi-Cevita transformation can be applied to show that solutions of the $n$-body problem can be analytically continued through isolated binary collisions (this is also known as regularization in the context of celestial mechanics). Interestingly, Mcgehee \cite{mcgehee1974a} used the blowup method to study the more complicated  triple collision in the context of the collinear three-body problem. Indeed, the author blew up the collision set to a sphere and upon applying desingularization, he obtained hyperbolic equilibria points on a collision manifold. This led the author (through hyperbolic invariant manifolds) to conclude that the triple collision cannot be regularized. Subsequently, this approach was used in \cite{elbialy1990a} to show that the simultaneous binary collision scenario was $C^0$-regularizable. Interestingly,  \cite{mart1999a} proved that it is exactly $C^{8/3}$-regularizable in the collinear case. A separate geometric prove -- based upon normal form theory and the blowup approach of \cite{elbialy1990a,mcgehee1974a} --  was given in \cite{duignan2020a}. This approach led the same authors in \cite{duignan2021a} to prove that the same results hold in the planar case, a result that was initially conjectured by   \cite{mart1999a}.

\subsection{Outline}
The paper is organized as follows: In \secref{existence}, we lay out our approach (based upon certain blowup transformations) for characterizing $\mathcal E_\infty$ and present two theorems on the existence of smooth invariant manifolds of \eqref{keplerd}: \thmref{main1} for the existence of the orbit corresponding to $\mathcal E_\infty=0$ and \thmref{main2} for the existence of a smooth invariant manifold corresponding to $\vert \mathcal E_\infty\vert \in (0,1)$, see also \thmref{main3}. We prove \thmref{main1} in \secref{main1} using general results on Gevrey-1 invariant manifolds $y=Y(x)$ for analytic systems of the form $x^2 \frac{dy}{dx} = F(x,y)$, with $F(0,0)=0$ and $D_y F(0,0)$ non-singular. It is author's impression that these results are not so well-known. We follow \cite{bonckaert2008a}, which proves the existence of such manifolds (and certain normal forms) in perhaps the most accessible way.  \thmref{main2} is proven in \secref{main2} using normal form theory (based upon averaging) to set up an appropriately regular equation for the invariant manifolds that can be solved upon application of the implicit function theorem. This approach may have general interest. In \secref{blowup}, we apply a sequence of blowup transformations along with an appropriate compactification in order to provide a geometric description of the dynamics. We summarize this in \figref{final}. From \cite{margheri2017a}, it is known that $\vert \mathcal E_\infty\vert$ cannot exceed $1$. The results of our blowup analysis, will provide a different characterization of this fact that also allow us to address the subtleties of $\vert \mathcal E_\infty\vert\rightarrow 1$. We lay this out in further details in our final discussion section, \secref{disc}.

\begin{figure}[h!]
 	\begin{center}
 		{\includegraphics[width=.9\textwidth]{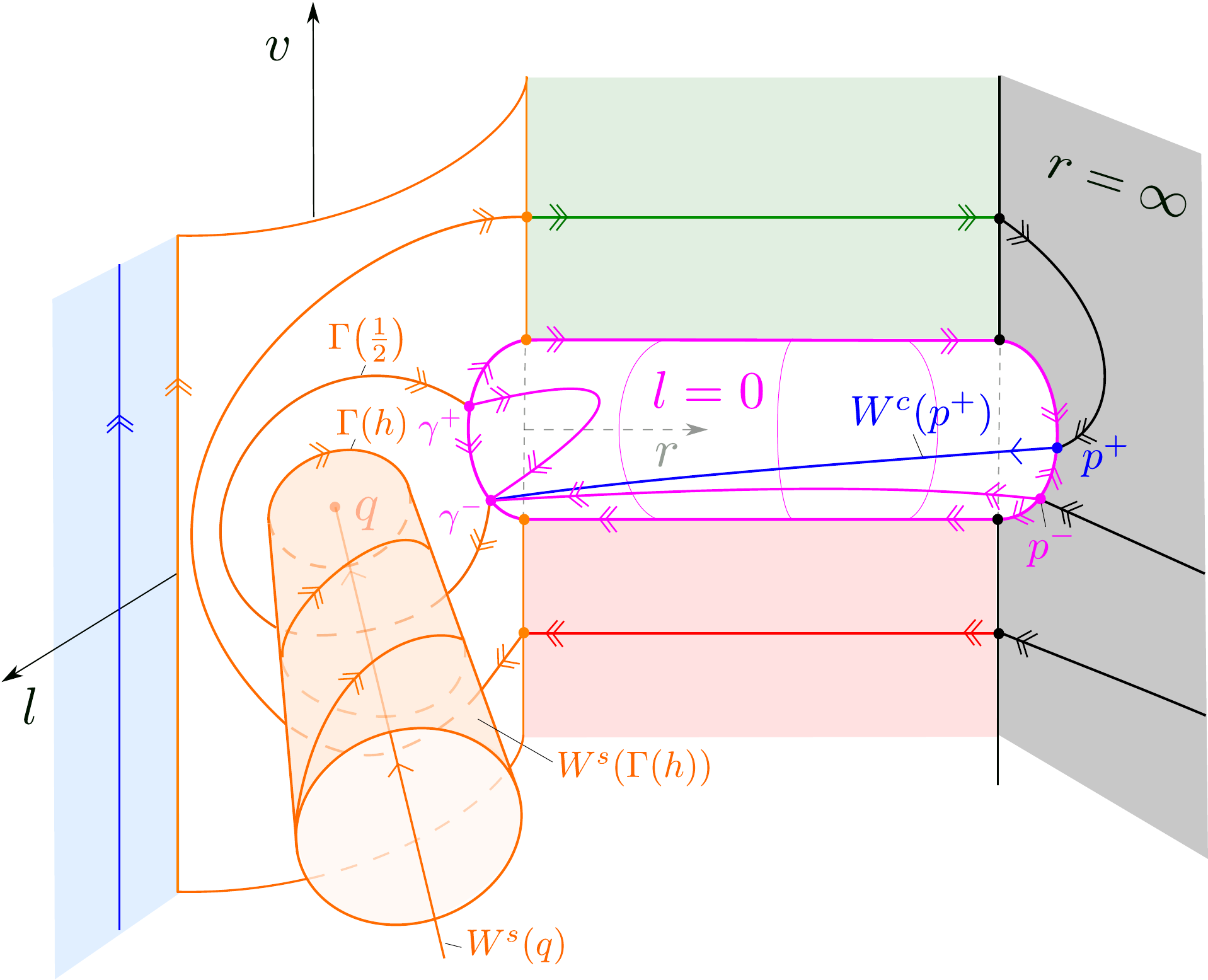}}
 		\caption{A geometric picture of the dynamics of \eqref{keplerd} upon blowup (and desingularization). Here $l=\vert L\vert$ and $v$ is defined by $v=l\dot r$, see \eqref{r1veqn}. The invariant manifolds $W^s(q)$ and $W^s(\Gamma(h))$ are stable sets of a zero-Hopf point $q$ and periodic orbits $\Gamma(h)$, $h\in (0,\frac12)$, respectively, on the blowup of $r=l=0$; the direction normal to $\Gamma(h)$, where $l>0$, is nonhyperbolic.  $W^s(\Gamma(h))$ corresponds to $\vert \mathcal E_\infty\vert \in (0,1)$ whereas $W^s(q)$ corresponds to $\mathcal E_\infty =0 $. Finally, $\vert \mathcal E_\infty\vert=1$ on the $l=0$-cylinder in purple. See \secref{blowup} for further details. The heteroclinic cycle $\Gamma_1\!\left(\frac12\right)$ is important for the description of $\vert \mathcal E_\infty\vert\rightarrow 1$, see \secref{disc}. Moreover, the heteroclinic orbits within $l=0$ connecting $\gamma^+$ and $\gamma^-$ correspond to ejection-collision orbits with $r(t)\rightarrow 0$ in backwards and forward time, see \cite{margheri2014a}. On the other hand, the connections between $p^-$ at infinity and $\gamma^-$ are capture-collision orbits with $r(t)\rightarrow 0$ in backwards and forward time, see e.g. \cite{diacu1999a,margheri2014a}. The unique capture-collision connection between the nonhyperbolic saddle $p^+$ and $\gamma^-$ is the boundary between ejection-collision and capture-collision orbits, see \remref{Wc}. }\figlab{final}
 	\end{center}
 \end{figure}

\section{Existence of invariant manifolds}\seclab{existence}
Since the direction of $L$ is preserved, it is without loss of generality to consider $u(t)\in \mathbb R^2$. We will do so henceforth. 
Upon identifying $\mathbb R^2$ with $\mathbb C$ in the usual way, we put 
\begin{align}
u:=re^{i\theta},\eqlab{ueqn}
\end{align}
and let  
\begin{align}\eqlab{leqn}
l:=\vert L\vert =r^2\dot \theta\ge 0,
\end{align}
denote the magnitude of the angular momentum. Then \eqref{keplerd} becomes
\begin{equation}\eqlab{rttlt}
\begin{aligned}
 \ddot r &=-\frac{1}{r^2}+\frac{l^2}{r^3}-\delta \dot r,\\
 \dot l&=-\delta l.
\end{aligned}
\end{equation}

For the purpose of this section, we suppose that $l>0$, and define the coordinates $r_1$ and $v$ by 
\begin{align}
r=l^2 r_1,\quad  v=l\dot r\eqlab{r1veqn}
\end{align} 
These coordinates appear in a systematic way through our blowup approach, see \secref{blowup}.
Then the kinetic energy, $K$, takes the following form:
\begin{equation}\nonumber
\begin{aligned}
 K(\dot u) &= \frac12  \left(\dot r^2 +\frac{l^2}{r^2}\right)=\frac{1}{2l^2} \left(v^2 +\frac{1}{r_1^2}\right),
\end{aligned}
\end{equation}
in the $(r_1,v,l)$-coordinates.
Moreover, we have the following
differential equations
\begin{equation}\eqlab{l1eqns}
\begin{aligned}
 \dot r_1 &=l^{-3} \left(v+2\delta r_1 l^3\right),\\
 \dot v &= l^{-3} \left(-\frac{r_1-1}{r_1^3}-2\delta vl^3\right),\\
 \dot l &=-\delta l.
 \end{aligned}
\end{equation}
Now, since $l^{-3}$ is a common factor of the right hand side, we make a nonlinear transformation of time that corresponds to multiplication by $l^3$:
\begin{equation}\eqlab{l1eqns2}
\begin{aligned}
 r_1' &=v+2\delta r_1 l^3,\\
 v' &= -\frac{r_1-1}{r_1^3}-2\delta vl^3,\\
 l' &=-\delta l^4.
\end{aligned}
\end{equation}
\eqref{l1eqns} and \eqref{l1eqns2} are equivalent for $l>0$. However, \eqref{l1eqns2} is defined for $l=0$ (it defines an invariant set) and we can therefore use dynamical systems theory to infer properties from $l=0$ to $l>0$ by working on \eqref{l1eqns2}. In turn, this then carries over to the equivalent system \eqref{l1eqns}. 

\begin{remark}\remlab{time}
 In the following we will reserve $\dot{()}$ to denote differentiation with respect to the original time. In comparison, we will use $()'$ repeatedly to refer to differentiation with respect to different times. It should be clear form the context how different $()'$ may be related. 
\end{remark}

Setting $l=0$ in \eqref{l1eqns2} gives
\begin{equation}\eqlab{r1veqns}
\begin{aligned}
 r_1' &= v,\\
 v' &=-\frac{r_1-1}{r_1^3},
\end{aligned}
\end{equation}
which is Hamiltonian with Hamiltonian function:
\begin{align}
 H(r_1,v) = \frac12 v^2+\frac{(r_1-1)^2}{2r_1^2}.\eqlab{H0func}
\end{align}

\begin{remark}\remlab{delta0}
If $\delta=0$, then \eqref{keplerd} reduces to the conservative Kepler problem and the Hamiltonian function \eqref{H0func}, written in terms of the original variables $r,\dot r$ and $l$:
\begin{align*}
 H(l^{-2} r,l\dot r) = \frac12 l^{2} \dot r^2+\frac{(l^2-r)^2}{2r^2},
\end{align*}
is, along with $l$, a conserved quantity of \eqref{rttlt}. In fact, $H=\frac12 \vert \mathcal E\vert^2$ in this case. 
\end{remark}

\begin{lemma}
Consider \eqref{ueqn}, \eqref{leqn} and \eqref{r1veqn}. Then 
\begin{align*}
 \dot u = \frac{v}{l}e^{i\theta}+\frac{il}{r}e^{i\theta},
\end{align*}
and 
\begin{align*}
 e^{-i\theta} \mathcal E(u,\dot u) = \frac{1-r_1}{r_1}-iv,
\end{align*}
so that
\begin{align}
H(r_1,v)=\frac12 \vert \mathcal E(u,\dot u)\vert^2.\eqlab{HvsMathcalE}
\end{align}

\end{lemma}
\begin{proof}
 Direct calculation.
\end{proof}

As a corollary, we have that 
\begin{align}
 H_\infty(r_1,v,l):=\lim_{t\rightarrow \infty}  H(\underline r_1(t,r_1,v,l),\underline v(t,r_1,v,l)) = \frac12\vert \mathcal E_\infty(u,\dot u)\vert^2,\eqlab{Hinf}
\end{align}
where we use $(\underline r_1(t,r_1,v,l),\underline v(t,r_1,v,l),\underline l(t,r_1,v,l))$ with
\begin{align*}
\underline z(0,\ldots)=z,
\end{align*}
 $z=r_1,v,$ and $l$, as our notation
for the flow of \eqref{l1eqns2}.

\begin{lemma}\lemmalab{Ham}
$(r_1,v)=(1,0)$ is a center for \eqref{r1veqns}, surrounded by periodic orbits $\Gamma_1(h)$, $h\in \left(0,\frac12\right)$, given by the level sets $H(r_1,v)=h$. $\Gamma_1(h)$  intersects the $r_1$-axis in two points $(r_{1,\pm}(h),0)$ with
\begin{align*}
r_{1,-}(h) = \frac{1}{1+\sqrt{2h}},\quad r_{1,+}(h) = \frac{1}{1-\sqrt{2h}}.
\end{align*}
The unbounded orbit $\Gamma_1\!\left(\frac12\right)$ given by the level set $H(r_1,v)=\frac12$ (corresponding to $\vert \mathcal E_\infty\vert=1$ by \eqref{Hinf}), intersects the $r_1$-axis once in $r_1=r_{1,-}(\frac12) =\frac12$, and is a separatrix, separating the bounded orbits ($H(r_1,v)<\frac12$) from unbounded orbits ($H(r_1,v)>\frac12$).
\end{lemma}
\begin{proof}
 Direct calculation. Notice in particular that $(r_1,v)=(1,0)$ is an extremum (minimum) of the Hamiltonian function and therefore also a center of \eqref{r1veqns}. In fact, the linearization around $(r_1,v)=(1,0)$ produces $\pm i$ as the eigenvalues.
\end{proof}
We illustrate the phase portrait of \eqref{r1veqns} in \figref{r1v}. 


\begin{figure}[h!]
 	\begin{center}
 		{\includegraphics[width=.6\textwidth]{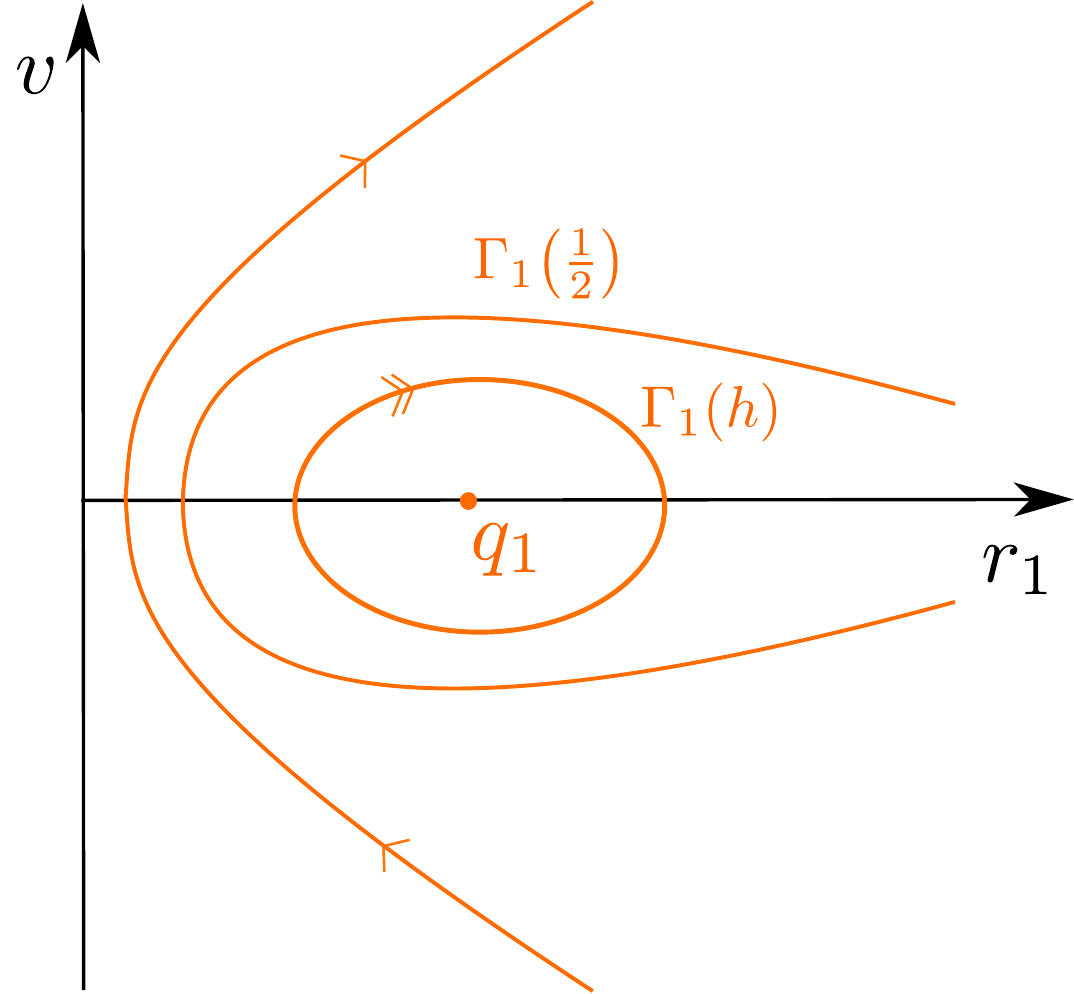}}
 		\caption{Phase portrait of \eqref{r1veqns}. The orbit $\Gamma_1\!\left(\frac12\right)$ defined by $H(r_1,v)=\frac12$ is a separatrix, separating bounded from unbounded orbits.}\figlab{r1v}
 	\end{center}
 \end{figure}


Linearization of \eqref{l1eqns2} at the corresponding equilibrium point $q_1$ of the full system, given by
\begin{align*}
 q_1:\quad (r_1,v,l)=(1,0,0),
\end{align*}
clearly produces eigenvalues $\pm i,0$. It therefore corresponds to a zero-Hopf point \cite{Guckenheimer97}. Similarly, all periodic orbits $\Gamma_1(h)$, $h\in \left(0,\frac12\right)$ are also degenerate when embedded within the full system \eqref{l1eqns2}. The description of the stable sets of these sets of points is therefore complicated by the lack of hyperbolicity. 

\begin{remark}
For $\delta=0$, recall \remref{delta0}, $q_1$ corresponds to circular orbits of the (conservative) Kepler problem with zero eccentricity, whereas $\Gamma_1(h)$, $h\in (0,\frac12)$, correspond to elliptic ones with $\vert \mathcal E\vert \in (0,1)$. Finally, $\Gamma_1\!\left(\frac12\right)$ corresponds to the parabolic orbit whereas $\Gamma_1(h)$, $h>\frac12$ corresponds to hyperbolic ones. 
\end{remark}
\subsection{Main results}
We now state our main results.
\begin{theorem}\thmlab{main1}
\textnormal{(The local stable manifold $W_{loc}^s(q_1)$ of $q_1$.)} The local stable set $W_{loc}^s(q_1)$ is a smooth one-dimensional manifold, taking the following graph form
\begin{align}
r_1 &=1+l^3 F_1(l^3),\quad v=
l^3 G_1(l^3),\quad l\in [0,l_0],\eqlab{manifold1}
\end{align}
for $l_0>0$ small enough and 
where $F_1,G_1:[0,l_0]\rightarrow \mathbb R$ are Gevrey-1 smooth functions. Moreover,
\begin{align}
 \frac{d^{2n}}{dx^{2n}} F_1(0) = 0,\quad \frac{d^{2n+1}}{dx^{2n+1}} G_1(0)=0,\eqlab{FevenGodd}
\end{align}
for all $n\in \mathbb N_0$.

\end{theorem}

Upon using \eqref{r1veqn}, we find that $W^s(q_1)$ in \eqref{manifold1} takes the following form
\begin{align*}
 r = l^2 (1+l^3 F_1(l^3)),\quad \dot r  = l^2 G_1(l^3),
\end{align*}
with respect to $r$ and $\dot r$. These quantities are smooth functions of $l$ along this orbit and decay like $e^{-2\delta t}$. Since $F_1(0)=0$, see \eqref{FevenGodd} with $n=0$, we obtain from \eqref{Lt}, \eqref{leqn} and \eqref{r1veqn} that
\begin{align}
\frac{d\theta}{dl} = -\frac{1}{\delta l^4(1+l^3 F_1(l^3)}:= -\frac{1}{\delta l^4} \left(1+l^6 \widetilde F_1(l^3)\right),\eqlab{thiseqn}
\end{align}
and consequently that the sum
\begin{align*}
l^3 \theta(l) +l^3 \int \frac{1}{\delta l^4} dl = l^3 \theta(l)-\frac{1}{3\delta},
\end{align*}
being equal to
\begin{align*}
 -\frac{l^3}{3\delta} \int \widetilde F_1(l^3) dl^3,
\end{align*}
by \eqref{thiseqn},
is a smooth function of $l^3$ along the orbit in \eqref{manifold1}. This result complements results of \cite[Eq. (43)]{margheri2017a} which showed that $\limsup_{t\rightarrow \infty}$ and $\liminf_{t\rightarrow \infty}$ of $\dot \theta l^{-3}$ both exist whenever $\vert \mathcal E_\infty \vert \in (0,1)$.

%

We now turn our attention to the stable sets of  $W^s(\Gamma_1(h))$ with $h\in (0,\frac12)$.
\begin{theorem}\thmlab{main2}
\textnormal{(The local stable manifold of $\Gamma_1(h)$.)}
Fix $h\in \left(0,\frac12\right)$. Then the stable set $W^s(\Gamma_1(h))$ of $\Gamma_1(h)$ is defined by $H_\infty(r_1,v,l)=h$. 
 Moreover, fix any $k\in \mathbb N$. Then there exists a neighboorhood $N(k,h)$ of $\Gamma_1(h)$ such that $W_{loc}^s(\Gamma_1(h))=W^s(\Gamma_1(h))\cap N(k,h)$ is a $C^k$-smooth two-dimensional submanifold. The dependency on $h$ is also $C^k$-smooth.
\end{theorem}
Although the domain $N(k,h)$ depends upon $k$ and $h$, the flow is regular away from $l=0$ and consequently, we can globalize the local manifolds by application of the backward flow. Since the system is real analytic, it follows that these global stable manifolds are in fact $C^k$ also. Similarly, if we fix a compact interval $I\subset \left(0,\frac12\right)$ then we have a uniform description of all local manifolds within $N(k):=\cap_{h\in I} N(k,h)$ and by working on this set, our approach also shows that the stable manifolds are also $C^k$. In turn, seeing that $k\in \mathbb N$ is arbitrary, we obtain the following:
\begin{theorem}\thmlab{main3}
The global stable manifolds $W^s(q_1)$ and $W^s(\Gamma_1(h))$ of $q_1:\,(r_1,v,l)=(1,0,0)$ and $\Gamma_1(h)$, $h\in \left(0,\frac12\right)$, respectively, are each $C^\infty$. The dependency of $W^s(\Gamma_1(h))$ on $h$ is also $C^\infty$. 
\end{theorem}

Due to \eqref{Hinf}, we have obtained a complete description of the smoothness of $\vert \mathcal E_\infty(u,\dot u)\vert\in(0,1)$ which was left open in \cite{margheri2017a}.  It seems likely that the global stable manifolds of $\Gamma_1(h)$ are also Gevrey-1 (as the stable manifold $W^s(q_1)$ of $q_1:\,(r_1,v,l)=(1,0,0)$), but this would require better normal forms. \cite{bittmann2018a} considers such normal forms but the condition regarding the trace is violated in the present context. 

Finally, we remark that the existence of the invariant set $W^s(q_1)$ was proven more indirectly in \cite{margheri2017a} using degree theory. Although this approach does not address the smoothness, the result in \cite{margheri2017ab} holds true for a more general class of dissipations defined by $\delta=\delta(\vert u\vert)$, with $\delta(\vert u\vert)\ge c>0$. However, our results can be extended to this general case also. We leave this to the interested reader. 
\section{Proof of \thmref{main1}}\seclab{main1}
To prove \thmref{main1}, we consider \eqref{l1eqns2} 
%
%
in terms of 
\begin{align}\eqlab{xeqn}
x=l^3.
\end{align} This gives
\begin{equation}\eqlab{r1vxeqns0}
\begin{aligned}
 r_1' &=v+2\delta r_1 x,\\
 v'  &= -\frac{r_1-1}{r_1^3}-2\delta vx,\\
 x' &=-3 \delta x^2,
\end{aligned}
\end{equation}
or
\begin{equation}\eqlab{r1vxeqns}
\begin{aligned}
 x^2 \frac{dr_1}{dx} &=\frac{1}{3\delta}\left(-v-2\delta r_1 x\right),\\
 x^2 \frac{dv}{dx} &= \frac{1}{3\delta}\left(\frac{r_1-1}{r_1^3}+2\delta vx\right),
\end{aligned}
\end{equation}
upon eliminating time. To zoom in on $q_1$ at $r_1=1,v=0$, $x=0$, we perform a blowup transformation:
\begin{align}
(x,r_{11},v_{11})\mapsto \begin{cases}
 r_1&=1+ xr_{11},\\
 v&=x(v_{11}-2\delta),\end{cases}\eqlab{v11}
\end{align}
leaving $x$ fixed.
This gives the final system:
\begin{align}\eqlab{yeqn}
 x^2 \frac{d y}{dx}
 &= A y+  f(x,y),
\end{align}
setting $y=(r_{11},v_{11})$ and where
\begin{align*}
 A &= \begin{pmatrix}
      0 & -\frac{1}{3\delta}\\
      \frac{1}{3\delta} & 0
     \end{pmatrix},\\
     f(x,r_{11},v_{11}) &= \frac13 x\begin{pmatrix} -5 \delta^{-1} r_{11}\\
  2\delta-v_{11}-\delta^{-1} r_{11}^2\left(3+3x r_{11}+x^2r_{11}^2\right){(1+xr_{11})^{-3}}
 \end{pmatrix}.
\end{align*}
The eigenvalues of $A$ are $\pm \frac{i}{\sqrt{3\delta}}$ and $f$ is real analytic. It is standard, that there exists a formal series solution 
\begin{align}
 y(x) = \widehat Y(x):=\sum_{x=1}^\infty Y_n x^n,\eqlab{Yhat}
\end{align}
of \eqref{yeqn}, see e.g. \cite{bonckaert2008a} and also \cite{haraux2021a} for the details of the expansion in the present case. In particular, 
\begin{align*}
 Y_1 = \begin{pmatrix}
        -2\delta^2 \\
        0
       \end{pmatrix},\quad 
Y_2 = \begin{pmatrix}
        0 \\
        16\delta^3
       \end{pmatrix},
\end{align*} and by induction on $n$
\begin{align}
 Y_{2n-1} &= \begin{pmatrix}
        *\\
        0
       \end{pmatrix},\quad 
Y_{2n} = \begin{pmatrix}
        0\\
        *
       \end{pmatrix},\eqlab{evenodd}
\end{align}
for all $n\in \mathbb N$. Here $*$ is an unspecified quantity that depends upon $n$ and $\delta$. 

 Let $S(\phi,r)\subset \mathbb C$ be the open sector region in $\mathbb C$ centered along the positive real axis, with radius $r$ and opening $\phi$:
\begin{align*}
 S(\phi,r) = \{x\in \mathbb C\,:\,\vert \text{arg}(x)\vert < \phi/2,0<\vert x\vert <r\}.
\end{align*}
Let $\overline S(\phi,r)$ denote its closure. 
The following result shows that the series \eqref{Yhat} is Gevrey-1 and \textit{1-summable} \cite{balser1994a}:
\begin{proposition}\proplab{Yexistence}
There exists $x_0>0$, $\phi>0$ both sufficiently small and a Gevrey-1 function $Y:\overline S(\phi+\pi,r)\rightarrow \mathbb C$ with $Y(0)=0$, which is real analytic on $S(\phi+\pi,r)$, such that the graph
\begin{align*}
 y = Y(x),\quad x\in S(\phi+\pi,r)
\end{align*}
solves \eqref{yeqn} with $Y(0)=0$. Here $Y$ has $\widehat Y$, see \eqref{Yhat}, as a Gevrey-1 asymptotic series, such that $Y^{(n)}(0)=n! Y_n$
 \end{proposition}
 \begin{proof}
  The result follows from \cite[Theorem 3]{bonckaert2008a}; although this result does not address the existence of an invariant manifold directly (instead \cite{bonckaert2008a} proves existence of a certain normal form), this can be obtained as a corollary, using the invariance of the set $y=0$ of \cite[Eq. (8)]{bonckaert2008a}. For completeness, we include a version of the proof that only addresses the existence of $y=Y(x)$, see also \cite[App. A]{kriszm1}, which includes a similar proof.
  
  Following
  \cite{bonckaert2008a} we proceed by first (a) transforming \eqref{r1vxeqns} into an equation on the ``Borel-plane'' (through the Borel-transform $\mathcal B$), then (b) apply a fixed-point argument there and finally (c) obtain our desired solution by applying the Laplace transform. 

For our purposes, the Borel transform is defined in the following way: If $h(x)=\sum_{n=1}^\infty h_n x^n$ is a Gevrey-1 formal series: 
\begin{align}
\vert h_n\vert \le a b^n n!,\eqlab{hncond}
\end{align}
then the Borel transform of $h$ is given by
\begin{align*}
 \mathcal B(h)(u) = \sum_{n=0}^\infty \frac{h_{n+1}}{n!} u^n.
\end{align*}
Clearly, $\mathcal B(h)$ is analytic on $\vert u\vert <b^{-1}$ if \eqref{hncond} holds true. For the Laplace transform, on the other hand, we need analytic functions $\alpha(u)$ that are at most exponentially growing $\vert \alpha(u)\vert \le \mathcal O(1) e^{\zeta\vert u\vert}$, $\zeta>0$, in an infinite sector. With this in mind, let $S(\phi)\subset \mathbb C$ be the (infinite) sector centered along the positive real $x$-axis with opening $\phi\in (0,\pi)$, and let $B(R)$ be the open ball of radius $R$ centered at $0$. Finally, set 
\begin{align*}
 \Delta := S(\phi)\cup B(R).
\end{align*}
Then for any $\zeta>0$,
we define the norm 
\begin{align}
\Vert \alpha\Vert_\zeta:=\sup_{u\in \Delta} \left\{\vert \alpha(u)\vert (1+\zeta^2 \vert u\vert^2)e^{-\zeta \vert u\vert}\right\},\eqlab{normG}
\end{align}
see \cite{bonckaert2008a}, 
on the space of analytic  functions on $\Delta$:
\begin{align*}
  \mathcal G:=\{\alpha:\alpha \mbox{  is analytic on  } \Delta \mbox{ and } \Vert \alpha\Vert_\zeta<\infty\}.
\end{align*}
The normed space $(\mathcal G,\Vert \cdot\Vert_\zeta)$ is a complete space. We will need $\zeta$ sufficiently large in the following.
The factor $1+\zeta^2 \vert w\vert^2$ in the norm $\Vert \cdot \Vert_{\zeta}$ ensures that the convolution:
 \begin{align*}
  (\alpha \star \beta)(u) = \int_0^u \alpha(s) \beta(u-s)ds,
 \end{align*}
 is continuous as a bilinear operator on $\mathcal G$. In particular, we have
\begin{align*}
 \Vert \alpha\star \beta \Vert_\zeta \le \frac{4\pi}{\zeta}\Vert \alpha\Vert_\zeta \Vert \beta\Vert_\zeta,
\end{align*}
 see \cite[Proposition 4]{bonckaert2008a}.
The Laplace transform (along the positive real axis)
\begin{align*}
 \mathcal L(\alpha)(x) :=\int_0^\infty \alpha(u) e^{-u/x} du,
\end{align*}
is then well-defined for any $\alpha\in \mathcal G$. In fact, we have the following result. 
\begin{lemma}\cite[Proposition 3]{bonckaert2008a}\lemmalab{A1}
 The Laplace transform defines a linear continuous mapping, with operator norm $\Vert \mathcal L\Vert\le 1$, from $\mathcal G$ to the set of analytic functions on a local sector $S(\pi+\phi,R_0)=S(\pi+\phi)\cap B(R_0)$ for $\phi\in (0,\pi)$ and $R_0>0$ sufficiently small. Moreover, 
 \begin{align}
  \mathcal L(\alpha \star \beta)(x) &= \mathcal L(\alpha)(x)\mathcal L(\beta)(x),\nonumber
  \end{align}
  and
  \begin{align}
  x^2\frac{d}{dx}\mathcal L(\alpha)(x)&=\mathcal L(u\alpha)(x),\eqlab{x2Lx}
 \end{align}
 with $u\alpha$ being the function $u\mapsto u\alpha(u)$, 
for every $\alpha,\beta\in \mathcal G$.
\end{lemma}

Following \eqref{x2Lx}, we are now led to write the left hand side of \eqref{yeqn} with $y=Y(x)$ as $[u I -A]\Phi(u)$ with $\Phi = \mathcal B(Y)$. 
To set up the associated right hand side, we need to deal with the nonlinearity $f(x,Y(x))$. This is described in \cite[Proposition 5]{bonckaert2008a}:
\begin{lemma}
Write $f$ as the convergent series $f(x,y)=\sum_{n=1}^\infty f_{n}(x) y^n$, $f_n(x):=\sum_{m=1}^\infty f_{mn}x^m$ and let $F_n$ be the Borel transform of $f_n$. Then $F_n\in \mathcal G$ for each $n$. Fix $C_0>0$ and for large values of $\zeta$, recall \eqref{normG}, consider $\alpha\in \mathcal G$ with $\Vert \alpha\Vert_\zeta\le C_0$. Then $\alpha\mapsto f^*(\alpha)$ defined by
\begin{align*}
f^*(\alpha)(u):=\sum_{n=1}^\infty F_n(u) \star \alpha(u)^{\star n},
\end{align*}
which converges in $\mathcal G$, is differentiable and satisfies the following estimates
\begin{align}
 \Vert f^*(\alpha)\Vert_{\zeta}\le C_1,\quad 
 \Vert D(f^*)(\alpha)\Vert_{\zeta}\le \zeta^{-1} C_1,\eqlab{fstarest}
\end{align}
for some constant $C_1>0$ depending only on $f$ and $C_0$. 
Moreover,
\begin{align}
 \mathcal L(f^*(\alpha))(x) = f(x,\mathcal L(\alpha)(x)).\eqlab{Lf}
\end{align}

\end{lemma}
Following \eqref{Lf}, we are therefore finally led to consider 
\begin{align}
 \Phi(u) = [u I -A]^{-1} f^*(\Phi)(u),\eqlab{borelEq}
\end{align}
where $\Phi$ is the Borel transform of $Y$. The equation \eqref{borelEq} has the form of a fixed point equation. Since the eigenvalues of $A$ are imaginary, if we take $\phi\in (0,\pi)$ then $[uI-A]^{-1}$ is uniformly bounded on $S(\phi)$. 
%
Using \eqref{fstarest}, it therefore follows that there is some $M>0$, depending on $f$, $\phi$ and $R$, such that the right hand side of \eqref{borelEq} defines a contraction on the subset of $\mathcal G$ with $\Vert \cdot \Vert_\zeta\le M$ for $\zeta>0$ large enough. Consequently, by Banach's fixed point theorem there is a unique solution $\Phi\in \mathcal G, \Vert \Phi \Vert_\zeta\le M$, solving \eqref{borelEq}. By applying the Laplace transform, we obtain the desired solution
\begin{align*}
 Y(x) := \mathcal L(\Phi)(x),
\end{align*}
 of \eqref{yeqn},
using \eqref{x2Lx} and \eqref{Lf}. The function $Y$ is defined on the domain $S(\pi+\phi,R_0)=S(\pi+\phi)\cap B(R_0)$ and has the properties specified by \lemmaref{A1}. Finally, we emphasize that the Borel transform is real when the argument is. Consequently, \eqref{borelEq} is real when $\Phi$ is real. Since the Laplace transformation is real upon integrating along the positive real axis, $Y=\mathcal L(\Phi)$ is real analytic on $S(\pi+\theta,R_0)$ as claimed.
 \end{proof}
Upon transforming the manifold in \propref{Yexistence} back to the $(r_1,v,l)$-coordinates, using \eqref{xeqn} and
\begin{align*}
 \begin{pmatrix}
  F_1(l^3)\\
  G_1(l^3)+2\delta
 \end{pmatrix} := Y(l^3),
\end{align*}
cf. \eqref{v11}, 
we obtain \eqref{manifold1}. \eqref{FevenGodd} then also follows from \eqref{evenodd}. $W_{loc}^s(q_1)$ cannot contain points not in \eqref{manifold1}; this follows from \thmref{main2}, which we prove in the following section. This completes the proof of \thmref{main1}.
 

\section{Proof of \thmref{main2}}\seclab{main2}
To prove \thmref{main2}, we will proceed in three steps: First we introduce appropriate (action-angle) coordinates to parameterize $H(r_1,v)=h$, see \secref{41}. Subsequently in \secref{42}, we bring our system into a normal form (based upon averaging). Finally in \secref{43}, we use the normal form to set up an equation for the invariant manifold $W^s(\Gamma_1(h))$, which we solve using the implicit function theorem. Essentially, our approach is reminiscint of a flow-box argument; we will show that there are smooth coordinates $(h,\phi,l)$, $\phi\in \mathbb T=\mathbb R/(2\pi \mathbb Z)$, with $h'=0$. 

\subsection{Action-angle coordinates}\seclab{41}
We first consider the planar Hamiltonian system \eqref{r1veqns}, repeated here for convinience:
\begin{equation}\eqlab{r1veqns2}
\begin{aligned}
 r_1' &= v,\\
 v' &=-\frac{r_1-1}{r_1^3},
\end{aligned}
\end{equation}
with Hamiltonian function:
\begin{align*}
 H(r_1,v) = \frac12 v^2 +\frac{(r_1-1)^2}{2r_1^2}.
\end{align*}
The set $H^{-1}(I)$ with $I := \left(0,\frac12\right)$ is filled with periodic orbits centered around the point $(r_1,v)=(1,0)$. It is standard, see e.g. \cite{meyer2009a}, that there exists a real-analytic symplectic diffeomorphism on this set
\begin{align*}
 \Psi:\,(r_1,v)\mapsto (A,\phi_1)\in (0,J_0)\times \mathbb T,\quad \mathbb T:=\mathbb R/(2\pi \mathbb Z),
\end{align*}
with action-angle coordinates $(A,\phi_1)$, such that \eqref{r1veqns2} becomes
\begin{align}
 A' &=0,\nonumber \\
 \phi_1' &=\Omega'(A).\eqlab{phidot}
\end{align}
Here $\Omega(A)=H\circ \Psi^{-1}(A,\phi_1)$ is the Hamiltonian function expressed in the new coordinates;  the important point is obviously that $\Omega$ is independent of $\phi_1$. Consider $\Gamma_1(h)=\{(r_1,v)\,:\,H(r_1,v)=h\}$, $h\in I$,  and write $\Psi =(\Psi_A,\Psi_\phi)$. Then $A=\Psi_A(r_1,v)$ is just the area of the oval $\{(r_1,v)\,:\,H(r_1,v)\le h\}$, whereas $\phi_1=\Psi_\phi(r_1,v)\in \mathbb T$ is so that $dr_1\wedge dv=d\phi_1 \wedge dA$.
By \eqref{phidot}, we see that the angle $\phi_1$ is a scaling of time ($\operatorname{mod}2\pi$) such that each $\Gamma_1(h)$ is $2\pi$-periodic. Consequently, if $p(h)$ is the period of $\Gamma_1(h)$ then 
\begin{align}
  \Omega'(A) = \frac{2\pi}{p(h)}>0.\eqlab{Omegaprime}
\end{align}
For our purposes, it will be more convinient to use $H=\Omega(A)$ rather than $A$ as an action-variable. Using \eqref{Omegaprime}, we have that 
\begin{align}
A\mapsto H=\Omega(A),\eqlab{H}
\end{align}
is a real-analytic diffeomorphism  with the inverse function $H\mapsto \Omega^{-1}(H)$ defined on $H\in \left(0,\frac12\right)$. We compose $\Psi_A$ with \eqref{H} and obtain the following real-analytic diffeomorphism:
\begin{align*}
 (r_1,v)\mapsto (H,\phi_1)\in \left(0,\frac12\right)\times \mathbb T,
\end{align*}
on $H^{-1}(I)$, transforming \eqref{r1veqns2} into
\begin{align*}
 H'&=0,\\
 \phi_1' &=\Omega_0(H),
\end{align*}
with
\begin{align}
 \Omega_0(H):=\Omega'(\Omega^{-1}(H)).\eqlab{Omega0H}
\end{align}

We now turn our attention to the $l>0$-system. We will again prefer to work with $x=l^3$ rather than $l$ and therefore consider \eqref{r1vxeqns0}, repeated here for convinience
\begin{equation}\eqlab{r1vxeqns02}
\begin{aligned}
 r_1' &=v+2\delta r_1 x,\\
 v'  &= -\frac{r_1-1}{r_1^3}-2\delta vx,\\
 x' &=-3 \delta x^2.
\end{aligned}
\end{equation}
\begin{lemma}\lemmalab{tildeH}
Let $X_1$ denote the vector-field of \eqref{r1vxeqns02}
and define the following function
\begin{align*}
 H_1(r_1,v,x):=H(r_1,v) +2 \delta r_1 v x +3 \delta^2 r_{1}^2 x^2.
\end{align*}
Then 
\begin{align}
 \mathcal L_{X_1} H_1(r_1,v,x) &=-6\delta^3 r_1^2 x^3,\eqlab{H1dot}
\end{align}
where $\mathcal L_{X_1} H_1=\nabla H_1 \cdot X_1$ is the Lie-derivative. 
\end{lemma}
\begin{proof}
Follows from a direct calculation:
 \begin{align*}
\frac{\partial H_1}{\partial r_1}r_1'+\frac{\partial H_1}{\partial v}v'&= 6 \delta^2 r_1 v x^2 + 12 \delta^3 r_1 x^3,\quad
\frac{\partial H_1}{\partial x}x' = -6 \delta^2 r_1 v x^2 - 18 \delta^3 r_{1}^2 x^3,
 \end{align*}
so that 
\begin{align*}
 \mathcal L_{X_1} H_1 = \frac{\partial H_1}{\partial r_1}r_1'+\frac{\partial H_1}{\partial v}v'+\frac{\partial H_1}{\partial x}x' =-6\delta^3 r_1^2 x^3.
\end{align*}

\end{proof}

\begin{lemma}\lemmalab{tildePsi}
Fix an open interval $J$ such that $\overline J\subset I$. Then there exists an $x_0>0$ such that 
\begin{align*}
 \widetilde \Psi: (r_1,v,x)\mapsto (
                                 H_1,
                                 \phi_1,
                                 x
                                  )
\end{align*}
defined on $H^{-1}(J)\times (-x_0,x_0)$, is a real analytic diffeomorphism.
\end{lemma}
\begin{proof}
 $\widetilde \Psi(r_1,v,x) = (\widetilde \Psi_1(r_1,v,x),\widetilde \Psi_2(r_1,v,x),x)$ is an $x$-fibered (polynomial) perturbation of the diffeomorphism defined by $(r_1,v,x)\mapsto (H(r_1,v_1),\phi_1(r_1,v_1),x)$. The result therefore follows by the inverse function theorem.
\end{proof}

\begin{lemma}
Let $X_1$ denote the vector-field of \eqref{r1vxeqns02} and consider $\widetilde \Psi$ from \lemmaref{tildePsi}. Then $\widetilde \Psi_* X_1$ takes the following form:
 \begin{equation}\eqlab{Hphixeqns}
\begin{aligned}
 {H}_1' &= x^3 R_1(H_1,\phi_1,x),\\
 \phi_1' & = \Omega_0(H_1)+x P_1(H_1,\phi_1,x),\\
 x' &= -3\delta x^2,
\end{aligned}
\end{equation}
with $\Omega_0$ defined in \eqref{Omega0H}, for all $H_1\in J$. The functions $R_1$ and $P_1$ are both real-analytic functions defined  on $J\times \mathbb T \times (-x_1,x_1)$ for $x_1>0$ sufficiently small. 
\end{lemma}
\begin{proof}
 The result follows directly from \eqref{phidot}, \lemmaref{tildeH} and \lemmaref{tildePsi}. In particular, by \eqref{H1dot} we have that 
 \begin{align*}
  R_1(H_1,\phi_1,x) = -6\delta^3 r_1^2,
 \end{align*}
with $r_1=r_1(H_1,\phi_1,x)$, cf. \lemmaref{tildePsi}.
\end{proof}


%


\subsection{A normal form}\seclab{42}
We will now use an averaging approach to normalize \eqref{Hphixeqns}. This will consist of pushing the angle-dependency to higher order with respect to $x$. 

\begin{lemma}\lemmalab{Iterative}
 \textnormal{(The Iterative Lemma)}
 Consider the real-analytic system 
 \begin{equation}\eqlab{Hnphixeqns}
\begin{aligned}
 {H}_n' &= x^3 \Lambda_n(H_n,x)+x^{n+2} R_n(H_n,\phi_n,x),\\
 \phi_n' & = \Omega_n(H_n,x)+x^{n} P_n(H_n,\phi_n,x),\\
 x' &= -3\delta x^2,
\end{aligned}
\end{equation}
defined on $J_n\times \mathbb T\times (x_n,x_n)$ and with $n\in \mathbb N$. Here $\Omega_n(H_n,0)=\Omega_0(H_n)\ge c>0$. Then for any $J_{n+1}\subset J_n$ there exists a constant $x_{n+1}>0$ and an $x$-fibered, real-analytic diffeomorphism $(H_n,\phi_n,x)\mapsto (H_{n+1},\phi_{n+},x)$ of the near-identify form
\begin{align}
 H_{n+1} &= H_n + x^{n+2} T_n(H_n,\phi_n,x),\eqlab{Hn1}\\
 \phi_{n+1}&=\phi_n+x^{n} Q_n(H_n,\phi_n,x),\eqlab{phin1}
\end{align}
such that
 \begin{equation}\eqlab{Hn1phixeqns}
\begin{aligned}
 {H}_{n+1}' &= x^3 \Lambda_{n+1}(H_{n+1},x)+x^{n+3} R_{n+1}(H_{n+1},\phi_{n+1},x),\\
 \phi_{n+1}' & = \Omega_{n+1}(H_{n+1},x)+x^{n+1} P_{n+1}(H_{n+1},\phi_{n+1},x),\\
 x'&= -3\delta x^2,
\end{aligned}
\end{equation}
with the right hand side being real-analytic on $J_{n+1}\times \mathbb T\times (-x_{n+1},x_{n+1})$.
\end{lemma}
\begin{proof}
 The result follows from a modification of the classical averaging theorem, see e.g. \cite{Guckenheimer97}. We write 
 \begin{align*}
  R_n(H_n,\phi_n,x)  = \overline R_n(H_n,x)+\widetilde R_n(H_n,\phi_n,x),
 \end{align*}
where 
\begin{align*}
\overline R_n(H_n,x):= \frac{1}{2\pi}\int_0^{2\pi} R_n(H_n,s,x) ds,
 \end{align*}
 is the mean of $\phi\mapsto R_n(H_n,\phi,x)$, and where $\widetilde R_n$ has zero mean:
  \begin{align}
   \int_{0}^{2\pi} \widetilde R_n(H_n,s,x) ds = 0.\eqlab{tildeRzeromean}
  \end{align}
We then take 
\begin{align*}
T_n(H_n,\phi_n,x):=-\Omega_0(H_n)^{-1} \int_{0}^{\phi_n} \widetilde R_{n}(H_n,s,x)ds,
\end{align*}
in \eqref{Hn1}. Notice that $T_n$ is well-defined for $\phi_n \in \mathbb T$  by \eqref{tildeRzeromean}. Moreover,
\begin{align}
 \frac{\partial }{\partial \phi_n} T_n(H_n,\phi_n,x) = -\Omega_0(H_n)^{-1} \widetilde R_{n}(H_n,\phi_n,x).\eqlab{Tnpartial}
\end{align}
We therefore have by \eqref{Hn1}:
\begin{align*}
H_{n+1}' &= H_n' +x^{n+2} \frac{\partial}{\partial\phi} T_n(H_n,\phi_n,x)\phi_n'+\mathcal O(x^{n+3})\\
 &=x^3 \Lambda_n(H_n,x) +x^{n+2} \overline R_n(H_n,x)+x^{n+2}\left\{ \widetilde R_n(H_n,\phi_n,x) +\frac{\partial}{\partial \phi_n}T_n(H_n,\phi_n,x) \Omega_0(H_n) \right\}\\
 &+ \mathcal O(x^{n+3})\\
 &:=x^3 \Lambda_{n+1}(H_{n+1},x) + x^{n+3} R_{n+ \frac12}(H_{n+1},\phi_n,x),
\end{align*}
with 
\begin{align*}
 \Lambda_{n+1}(H_{n+1},x) :=\Lambda_n(H_{n+1},x)+x^n \overline R_n(H_{n+1},x),
\end{align*}
using \eqref{Tnpartial} to conclude that $\{\cdots\}=0$. 

Subsequently, we define
\begin{align*}
 P_{n+\frac12}(H_{n+1},\phi_n,x):=x^{-n} \left[\Omega_n(H_n,x)-\Omega_n(H_{n+1},x) +x^{n} P_n(H_n,\phi_n,x)\right],
\end{align*}
such that 
\begin{align*}
 \phi_n' = \Omega_n(H_{n+1},x)+x^{n}P_{n+\frac12}(H_{n+1},\phi_n,x).
\end{align*}
 It follows from \eqref{Hn1} that $P_{n+\frac12}$ extends smoothly to $x=0$. We write 
 \begin{align*}
  P_{n+\frac12}(H_{n+1},\phi_n,x)  = \overline P_{n+\frac12}(H_{n+1},x)+\widetilde P_{n+\frac12}(H_{n+1},\phi_n,x),
 \end{align*}
where 
\begin{align*}
\overline P_{n+\frac12}(H_{n+1},x):= \frac{1}{2\pi}\int_0^{2\pi} P_{n+\frac12}(H_{n+1},s,x) ds,
 \end{align*}
 is the mean of $\phi\mapsto P_{n+\frac12}(H_{n+1},\phi,x)$, and where $\widetilde P_{n+\frac12}$ has zero mean:
  \begin{align}
   \int_{0}^{2\pi} \widetilde P_{n+\frac12}(H_{n+1},s,x) ds = 0.\nonumber
  \end{align}
Then we apply the same procedure on the $\phi_n$-equation. In particular, we consider a transformation defined by
\begin{align*}
\phi_{n+1}&=\phi_n+x^{n}  Q_{n+\frac12}(H_{n+1},\phi_n,x),
\end{align*}
fixing $H_{n+1}$ and $x$, 
with
\begin{align}
 Q_{n+\frac12}(H_{n+1},\phi_n,x):=-\Omega_0(H_{n+1})^{-1} \int_{0}^{\phi_n} \widetilde P_{n+\frac12}(H_{n+1},s,x)ds.\eqlab{Qn1_2}
\end{align}
This leads to the following equation for $\phi_{n+1}$:
\begin{align*}
  \phi_{n+1}' &= \phi_n' + x^{n} \frac{\partial }{\partial \phi_n} Q_{n+\frac12}(H_{n+1},\phi_n,x)\phi_n' +\mathcal O(x^{n+1})\\
 &=\Omega_n(H_{n+1},x)+x^{n}\overline P_{n+\frac12}(H_{n+1},x)\\
 &+x^{n}\left\{\widetilde P_{n+\frac12}(H_{n+1},\phi,x)+\frac{\partial }{\partial \phi_n} Q_{n+\frac12}(H_{n+1},\phi_n,x)\Omega_0(H_{n+1})\right\}+\mathcal O(x^{n+1})\\
 &=:\Omega_{n+1}(H_{n+1},x) +x^{n+1} P_{n+1}(H_{n+1},\phi_{n+1},x),
\end{align*}
with
\begin{align*}
 \Omega_{n+1}(H_{n+1},x):=\Omega_n(H_{n+1},x)+x^{n}\overline P_{n+\frac12}(H_{n+1},x),
\end{align*}
using \eqref{Qn1_2} to conclude that $\{\cdots\}=0$. Finally, we put 
$$R_{n+1}(H_{n+1},\phi_{n+1},x):=R_{n+\frac12}(H_{n+1},\phi_n,x).$$ This completes the proof. 
\end{proof}

Since \eqref{Hphixeqns} satisfies the conditions of The Iterative Lemma with $n=1$ and $\Lambda_1(H_1,x)\equiv 0$, $\Omega_1(H_1,x)\equiv \Omega_0(H_1)$, we conclude that for every $J\subset I$ and every $N\in \mathbb N_0$, we can transform \eqref{Hphixeqns} into ``the normal form'':
\begin{equation}\eqlab{HNphixeqns0}
 \begin{aligned}
 {H}_{N+1}' &= x^3 \Lambda_{N+1}(H_{N+1},x)+x^{N+3} R_{N+1}(H_{N+1},\phi_{N+1},x),\\
 \phi_{N+1}' & = \Omega_{N+1}(H_{N+1},x)+x^{N+1} P_{N+1}(H_{N+1},\phi_{N+1},x),\\
 x' &= -3\delta x^2,
\end{aligned}
\end{equation}
on $J\times \mathbb T\times (-x_{N+1},x_{N+1})$ for $x_{N+1}>0$ sufficiently small, by a near-identify transformation of $H$ and $\phi$. We drop the subscripts henceforth and write it in the equivalent form
\begin{equation}\eqlab{HNphixeqns}
\begin{aligned}
 \frac{dH}{d\tau} &= \frac{1}{3\delta} x\left[ \Lambda(H,x)+x^{N} R(H,\phi,x)\right],\\
 \frac{d\phi}{d\tau} & = \frac{1}{3\delta x^2} \left\{\Omega(H,x)+x^{N+1} P(H,\phi,x)\right\},\\
 \frac{dx}{d\tau}&=-1,
\end{aligned}
\end{equation}
for $x>0$.

\begin{lemma}\lemmalab{H0phi0}
 The flow of \eqref{HNphixeqns}, $\underline H(\tau,H_0,\phi_0,x_0),\underline \phi(\tau,H_0,\phi_0,x_0)$, $\underline x(\tau,H_0,\phi_0,x_0)=x_0-\tau$ with 
 \begin{align*}
  \underline z(0,\cdot)=z,
 \end{align*}
for $z=H,\phi$, is $C^\infty$ on the set 
\begin{align}\eqlab{Vxi}
V(\xi):=\left\{(\tau,H_0,\phi_0,x_0) \in (0,\xi)\times J\times \mathbb T\times  (0,\xi)\,:\,0<\tau<x_0\right\},
\end{align} for $\xi>0$ sufficiently small. 
\end{lemma}
\begin{proof}
The right hand side is smooth, even real-analytic, and the existence of a smooth local flow $\underline H(\tau,H_0,\phi_0,x_0),\underline \phi(\tau,H_0,\phi_0,x_0),\underline x(\tau,H_0,\phi_0,x_0)=x_0-\tau$, with $\tau\in I(H_0,\phi_0,x_0):=(0,\tau_{\text{max}})$, therefore follows. 
We then integrate both sides of \eqref{HNphixeqns}:
\begin{align}
 \underline H(\tau,H_0,\phi_0,x_0) &=H_0-\frac{1}{3\delta}\int_{x_0}^x s \left[\cdots \right]ds,\eqlab{underlineH}\\
 \underline \phi(\tau,\phi_0,x_0)&=\phi_0 -\frac{1}{3\delta}\int_{x_0}^x s^{-2} \left\{\cdots \right\}ds,\nonumber
\end{align}
with $x=x_0-\tau$, and $[\cdots]$, $\{\cdots\}$ being the brackets in  \eqref{HNphixeqns} evaluated at $H=\underline H(x_0-x,H_0,\phi_0,x_0),\phi=\underline \phi(x_0-x,H_0,\phi_0,x_0)$.
It is then standard to arrive at the following estimate for $\tau\in I(H_0,\phi_0,x_0)$:
\begin{align}
 \vert \underline H(\tau,H_0,\phi_0,x_0)-H_0\vert \le  C\vert \tau\vert\le C x_0, \quad \vert \underline \phi(\tau,H_0,\phi_0,x_0) -\phi_0\vert \vert \tau-x_0\vert \le C,\eqlab{Hphibound}
\end{align}
for some $C>0$ large enough and all $(H_0,\phi_0,x_0)\in J\times \mathbb T\times (0,\xi)$, provided that $\xi>0$ is small enough. This shows that $\tau_{\text{max}}=x_0$ and completes the proof. 
%

%

\end{proof}

\subsection{The existence of smooth invariant manifolds}\seclab{43}
The inequalities in \eqref{Hphibound} provide $C^0$-estimates of $\underline H$ and $\underline \phi$, respectively. 
\begin{lemma}\lemmalab{contH0}
 $\underline H(\tau,H_0,\phi_0,x_0)$ extends continuously to the closure $\overline{V(\xi)}$, recall \eqref{Vxi}, with $\underline H(0,H_0,\phi_0,0)=H_0$ for all $(H_0,\phi_0)\in J\times \mathbb T$.
 \end{lemma}
\begin{remark}
 This is true even though $\underline \phi$ itself does not extend to $\tau=x_0$.
\end{remark}

 \begin{proof}
   Consider $\underline \phi(\tau,H_0,\phi_0,x_0)$ given for $\tau\in [0,x_0)$. It is continuous on $V(\xi)$. 
Moreover,
$(\tau,x_0)\mapsto \underline H(\tau ,H_0,\phi_0,x_0)$ is absolutely continuous, uniformly in $(H_0,\phi_0)\in J\times \mathbb T$, see \eqref{underlineH}. Consequently, $\underline H$ extends continuously and uniquely to the closure  $\overline{V(\xi)}$. Moreover, $\underline H(0,H_0,\phi_0,x_0)=H_0$ by definition for all $x_0\in (0,\xi)$ and therefore also $H(0,H_0,\phi_0,0)=H_0$.
 \end{proof}
Let $h\in J_0\subset J$. We then consider the resulting equation
\begin{align}
 h = \underline H(x,H,\phi,x),\eqlab{heqn1}
\end{align}
with $(H,\phi,x)\mapsto \underline H(x,H,\phi,x)$ being  defined and continuous by \lemmaref{contH0} on $J\times \mathbb T\times [0,\xi]$. 
\begin{lemma}
The equation \eqref{heqn1} defines an invariant set $W^s(\Gamma_1(h))$ in the $(H,\phi,x)$-space. 
\end{lemma}
\begin{proof} $(x_0,H_0,\phi_0)\in W^s(\Gamma_1(h))\Longrightarrow $
\begin{align*}
 h& = \lim_{\tau\rightarrow x_0^-}\underline H(\tau,H_0,\phi_0,x_0)\\
 &=\lim_{\tau\rightarrow (x_0-s)^-}\underline H(\tau,\underline H(s,H_0,\phi_0,x_0),\underline \phi(s,H_0,\phi_0,x_0),x_0-s)\\
 &=\underline H(x_0-s,\underline H(s,H_0,\phi_0,x_0),\underline \phi(s,H_0,\phi_0,x_0),x_0-s),
\end{align*}
using the group properties of the flow. Consequently, $(\underline H(s,H_0,\phi_0,x_0),\underline \phi(s,H_0,\phi_0,x_0),x_0-s)\in W^s(\Gamma_1(h))$ for all $s\in (0,x_0)$. 
\end{proof}
In terms of the time $s$ (say) used in \eqref{HNphixeqns0}, $\tau\rightarrow x_0^-$ corresponds to $s\rightarrow \infty$, and since $$\underline H(0,H_0,\phi_0,0)=\frac12 v^2+\frac{(r_1-1)^2}{2r_1^2},$$ recall \eqref{H0func}, see also \lemmaref{tildePsi} and \lemmaref{Iterative}, we conclude that $W^s(\Gamma_1(h))$ is the stable set of $\Gamma_1(h)$, $h\in J_0\subset (0,\frac12)$, as desired.
%

%

\begin{proposition}\proplab{this}
Consider \eqref{HNphixeqns} with $N\in \mathbb N$ fixed and suppose that there is an $M\in \mathbb N$ so that $\underline H$ extends as a $C^M$-smooth function to the closure $\overline{V(\xi)}$, recall \eqref{Vxi}, satisfying 
\begin{align}
 \frac{\partial}{\partial H_0} \underline H(0,H_0,\phi_0,0)=1.\eqlab{HH0}
\end{align}
Then for any $J_0\subset J$ there exists a $\xi>0$ such that the following holds. \eqref{heqn1} with $h\in J_0$ has a unique solution for $(H,\phi,x)\in J\times \mathbb T\times [0,\xi]$ of the following graph form 
 \begin{align}
  H = F(h,\phi,x),\eqlab{H0expr}
 \end{align}
with $F\in C^M$ on $J_0\times \mathbb T\times [0,\xi]$. 
\end{proposition}
\begin{proof}
 Follows directly from the implicit function theorem.  Indeed, we have
 \begin{align*}
  h = \underline H(0,h,\phi,0),\quad \frac{\partial}{\partial H_0}\underline H(0,h,\phi,0)=1,
 \end{align*}
for all $\phi\in \mathbb T$ by \eqref{HH0}, and the right hand side of \eqref{heqn1} is well-defined and $C^M$ with respect to $(H,\phi,x)\in J\times \mathbb T\times [0,x_0]$ by assumption. Consequently, we can solve \eqref{heqn1} for $H$ as a function of $h,\phi,x$. This gives \eqref{H0expr} and $F$ is $C^M$-smooth since $\underline H$ is so.
\end{proof}
Under the assumptions of \propref{this}, we then have the following by returning to $l=x^3$:
For any $h\in J_0$, \eqref{H0expr} parametrizes $W^s(\Gamma_1(h))$ locally in the $(H,\phi,l)$-space as a $C^M$-smooth graph $H=F(h,\phi,l^3)$, $\phi\in \mathbb T$, $l\in (0,\xi^{1/3})$. Now, by \lemmaref{tildePsi} and \lemmaref{Iterative}, it follows that $(r_1,v,l)\mapsto (H,\phi,l)$, with $l\in (0,\xi^{1/3})$ for $\xi>0$ small enough, is a smooth diffeomorphism (on the relevant set). Therefore we obtain a $C^M$ invariant manifold $W^s(\Gamma_1(h))$ as the stable set of $\Gamma_1(h)$ in the original $(r_1,v,l)$-space, as desired.

Consequently, in order to finish the proof of  \thmref{main2}, it suffices to verify the conditions of \propref{this} and to note that $M$ can be taken to be arbitrary (upon increasing $N$). We will show that we can take $M=\lfloor \frac{N}{2}\rfloor$. Here $\lfloor x\rfloor$ for $x\in \mathbb R$ is the floor function, i.e. $n=\lfloor x\rfloor$ is the largest integer such that $n\le x$. $\underline H$ clearly satisfies \eqref{HH0} once we have shown that it extends $C^M$-smoothly to $\overline{ V(\xi)}$.

Define
 \begin{align}\eqlab{znu}
  \underline z_{\mathbf \nu} := \frac{\partial^{\vert \mathbf \nu\vert} }{\partial \tau^{\nu_{1}} \partial H_0^{\nu_2}\partial \phi_0^{\nu_3} \partial x_0^{\nu_4} } \underline z,
 \end{align}
for $z=H,\phi,x$ and where $\mathbf \nu=(\nu_1,\nu_2,\nu_3,\nu_3)\in \mathbb N_0^4$, $\vert \nu\vert=\nu_1+\cdots +\nu_4\ge 0$. 
\begin{lemma}\lemmalab{est0}
Fix $N\in \mathbb N$ with $N\ge 3$ and let $M=\lfloor \frac{N}{2}\rfloor$. Then there exists a constant $C_N>0$ large enough, and a constant $\xi>0$ small enough such that 
\begin{align*}
 \vert \underline H_{\mathbf \nu} (\tau,H_0,\phi_0,x_0)\vert \le C_N,
\end{align*}
for $(\tau,H_0,\phi_0,x_0)\in \overline{ V(\xi)}$,
and all $\vert \nu\vert\le M$.
\end{lemma}
\begin{proof}
The proof is delayed to the \appref{est}. It rests upon careful estimation of the higher order variational equations of \eqref{HNphixeqns}. The main difficulty lies in estimating $\underline \phi_{\mathbf \nu}$, due to the singular nature of \eqref{HNphixeqns} at $x=0$. We find that 
 \begin{align*}
  \vert \underline \phi_{\mathbf \nu} (\tau,H_0,\phi_0,x_0)\vert \vert \tau-x_0\vert^{1+\nu_1+\nu_4} \le C_N,
 \end{align*}
 for $(\tau,H_0,\phi_0,x_0)\in  V(\xi)$,
and all $\vert \nu\vert\le M$. Upon using that the $\phi$-dependent term in the $H$-equation has a $x^{N+1}$-factor, this allow us to control and extend (as in the proof of \lemmaref{contH0}) $\underline H_{\mathbf \nu}$ to the closure $\overline{V(\xi)}$ provided that $N> \nu_1+\nu_4$.
 
\end{proof}


\section{Blowing up the linearly damped Kepler problem}\seclab{blowup}
 In \secref{existence}, we characterized constant values of $\vert \mathcal E_\infty\vert \in (0,1)$ as a smooth cylinder in the $(r_1,v,l)$-space, see \thmref{main2}. $\mathcal E_\infty=0$, on the other hand, became a one-dimensional manifold, see \thmref{main1}. At the same time, since  $\vert \mathcal E_\infty\vert=1\Rightarrow L=0$, see \cite{margheri2017a}, $\vert \mathcal E_\infty\vert=1$ is the plane defined by $(r,v,0)$ in the $(r,v,l)$-space. $\vert \mathcal E_\infty\vert=0$ and $\vert \mathcal E_\infty\vert=1$ are therefore special cases where the associated invariant sets bifurcate. From \cite{margheri2017a,margheri2017ab}, it is also known that $\vert \mathcal E_\infty\vert\in [0,1]$ and that all values are attained in this set. 
In our characterization of $\vert \mathcal E_\infty\vert$ through the Hamiltonian function $H$, see \eqref{HvsMathcalE} and \eqref{Hinf}, this means the following:
\begin{lemma}\lemmalab{Hhge12}
The set $\Gamma_1(h)$ defined by $H(r_1,v)=h,\rho_1=0$ with $h\ge \frac12$, is not an $\omega$-limit set. 
\end{lemma}
\begin{proof}
 $H(r_1,v)>\frac12$ is obvious since then $\vert \mathcal E_\infty\vert>1$, which would contradict \cite{margheri2017a}. Moreover, although $H(r_1,v)=\frac12$ corresponds to $\vert \mathcal E_\infty\vert =1$ it cannot be an $\omega$-limit set either, because of $\vert \mathcal E_\infty\vert =1\Rightarrow L=0$, cf. \cite{margheri2017a}, which is an invariant set (that does not contain $\Gamma_1(h)$). 
\end{proof}
We obtained this result as a corollary of \cite{margheri2017a}. It cannot be understood directly from the perspective in \secref{existence}. However, in this section, we will perform a thorough geometric description of the dynamics of \eqref{keplerd} within the orbital plane, i.e. $u=re^{i\theta}$, by using the blowup method (as well as desingularization and compactification). In this way, we obtain a system where all singularities have associated eigenvalues with nonzero real part (with the exception of $q$ which only has imaginary eigenvalues), and this allows us to interpret \lemmaref{Hhge12} in a separate geometric way, which will also shed light on $\vert \mathcal E_\infty\vert\rightarrow 1$ (corresponding to $h\rightarrow \frac12$ cf. \eqref{Hinf}), see also \secref{disc}. 

Our starting point for our blowup approach is to use the coordinates $(r,v,l)$, recall \eqref{r1veqn} and \eqref{leqn}. This produces the following system
\begin{equation}\eqlab{rvl}
\begin{aligned}
 \dot r &= \frac{v}{l},\\
 \dot v &=\frac{l^3}{r^3}-\frac{l}{r^2}-2\delta v,\\
 \dot l &=-\delta l. 
\end{aligned}
\end{equation}
We now define a reparametrization of time, corresponding to multiplication of the right hand side of \eqref{rvl} by $l r^3$:
\begin{equation}\eqlab{rvl2}
\begin{aligned}
 r' &= vr^3,\\
 v' &=l\left(l^3-rl-2\delta r^3 v\right),\\
 l' &=-\delta r^3 l^2.
\end{aligned}
\end{equation}
\eqref{rvl} and \eqref{rvl2} are equivalent for $r>0$, $l>0$, but \eqref{rvl2} has the advantage of being well-defined on $r=0,l=0$. This enables the use of dynamical systems theory to infer properties from $r=0,l=0$, to $r>0,l>0$ for \eqref{rvl2} and therefore also \eqref{rvl}.

\begin{figure}[h!]
 	\begin{center}
 		{\includegraphics[width=.45\textwidth]{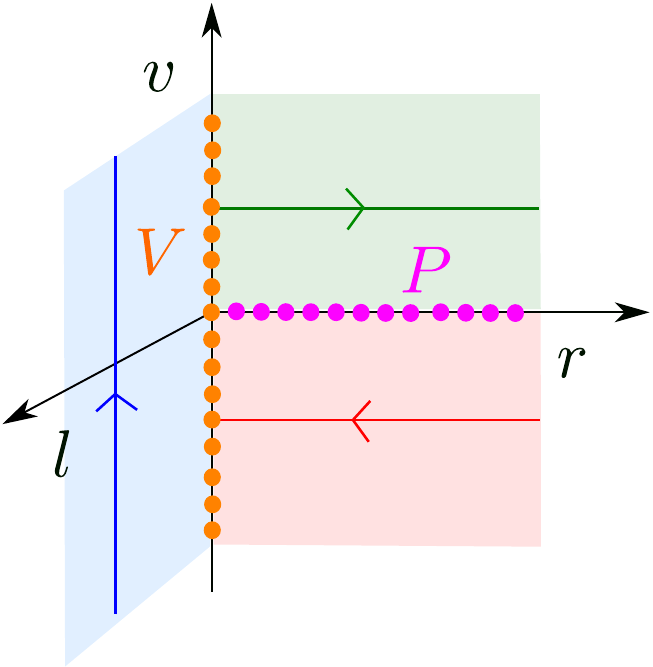}}
 		\caption{Dynamics of \eqref{rvl2}. $V$ and $P$ are sets of degenerate equilibria.}\figlab{rvl}
 	\end{center}
 \end{figure}

The system \eqref{rvl2} has two invariant planes: $\{r=0\}$ where
\begin{align*}
 v' &=l^4,\\
 l' &=0,
\end{align*}
and $\{l=0\}$ where
\begin{align*}
 r' &=vr^3,\\
 v' &=0,
\end{align*}
see \figref{rvl}. Their intersection $V=\{r=l=0\}$ and $P=\{v=l=0\}$ consist entirely of completely degenerate equilibria $v'=0$, insofar that the linearization about any point in $V$ or $P$ has only zero eigenvalues. We therefore apply the blowup method, see \cite{dumortier2006a,dumortier_1996}. We proceed as follows:
First, we blow up $V$ by application of the following cylindrical blowup transformation 
\begin{align}\eqlab{phiV}
\Phi^V:\quad  \rho\ge 0,(\bar r,\bar l)\in S^1\mapsto \begin{cases}
                                          r &= \rho^2 \bar r,\\
                                          l &=\rho \bar l,
                                         \end{cases}
\end{align}
which fixes $v$. See \figref{blowupV}. Let $X$ denote the vector-field in \eqref{rvl2}. Then the blowup weights are chosen so that $\overline X=\Phi^V_{*} X$ has $\rho^4$ as a common factor. It is therefore the desingularized vector-field $\rho^{-4}\overline X$, that we study in the following. To perform calculations, we work in two separate charts, that we will denote by $(\bar l=1)_1$ and $(\bar r=1)_2$, respectively, with chart-specific coordinates $(r_1,v,\rho_1)$ and $(\rho_2,v,l_2)$ defined by
\begin{align}
 (\bar l=1)_1:\quad \begin{cases}
                     r&=\rho_1^2 r_1,\\
                     l&=\rho_1,
                    \end{cases}\eqlab{c1}\\
 (\bar r=1)_2:\quad \begin{cases}
                     r&=\rho_2^2,\\
                     l&=\rho_2 l_2.
                    \end{cases}\eqlab{c2}                  
\end{align}
See also \figref{blowupV} for an illustration of these coordinates.
The following expressions 
\begin{align*}
 \rho_2 = \rho_1 \sqrt{r_1},\quad l_2 = 1/\sqrt{r_1},
\end{align*}
define the smooth change of coordinates
for $\rho_1\ge 0,r_1>0$. 
We will achieve the desingularization through division of the local vector-fields by $\rho_1^4$ and $\rho_2^4$, respectively.
\begin{remark}
  We used the $r_1$-coordinate of \eqref{c1} in \secref{existence}, see \eqref{r1veqn}, to prove \thmref{main1} and \thmref{main2}. The coordinates in the $(\bar r=1)_2$-chart allow us to follow and analyze the unbounded orbits $\Gamma(h)$ for $h\ge \frac12$. 
\end{remark}

\begin{figure}[h!]
 	\begin{center}
 		{\includegraphics[width=.9\textwidth]{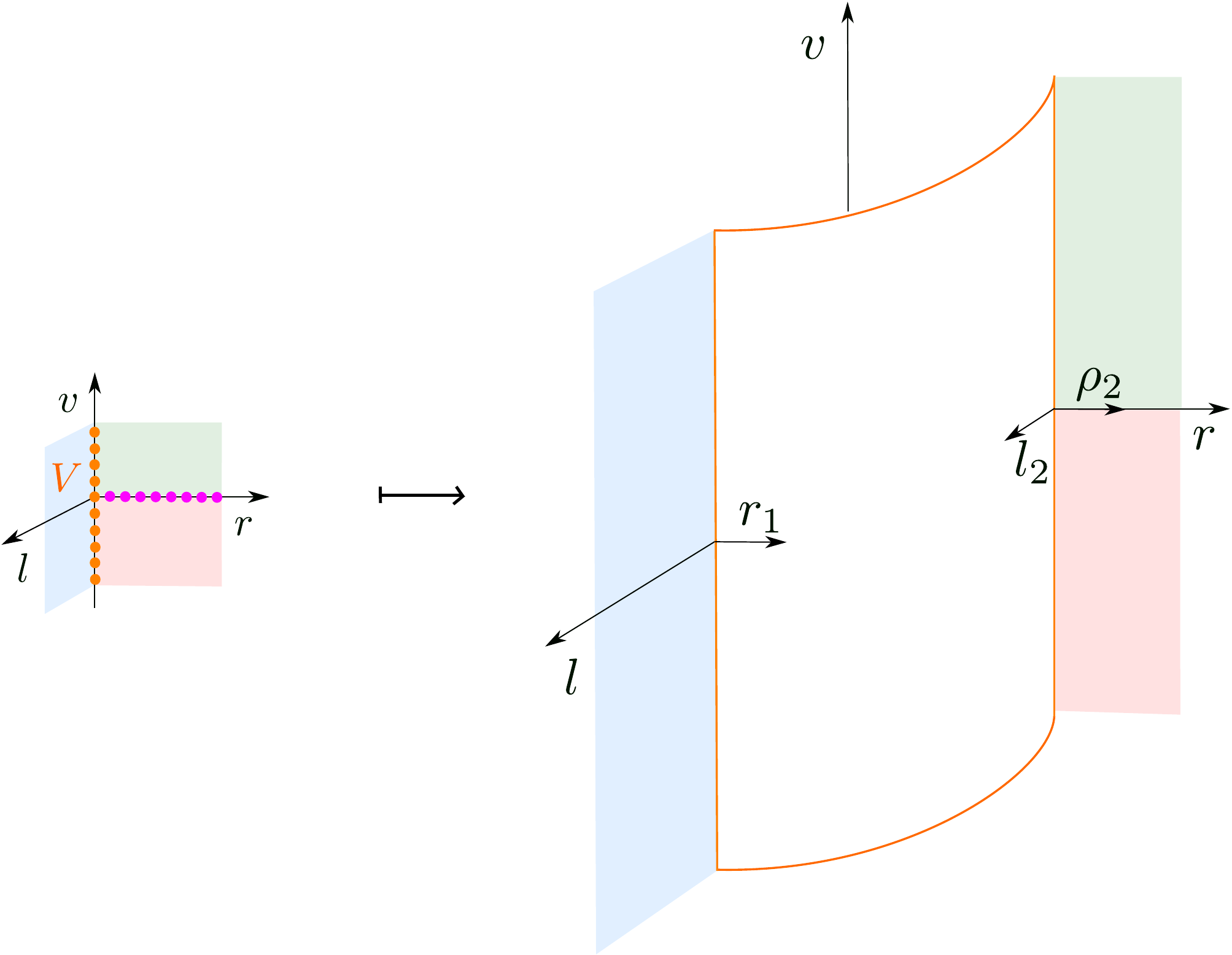}}
 		\caption{Illustration of the cylindrical blowup of the degenerate set $V$, see \eqref{phiV}.}\figlab{blowupV}
 	\end{center}
 \end{figure}

Subsequently, in the $(\bar r=1)_2$-chart, we will find that the set $P_2$ defined by
\begin{align*}
 v_2=0,\,l_2=0,
\end{align*}
and $\rho_2\ge 0$,
is a set of completely degenerate equilibria; this set clearly corresponds to $P$. We therefore blowup $P_2$ by application of the following cylindrical blowup transformation:
\begin{align}\eqlab{phiR2}
\Phi^{P_2}:\quad  \mu\ge 0,(\bar v,\bar l_2)\in S^1\mapsto \begin{cases}
                                          v &= \mu \bar v,\\
                                          l &=\mu \bar l_2,
                                         \end{cases}
\end{align}
leaving $\rho_2$ fixed. See \figref{blowupVR}. Let $X_2$ denote the local vector-field in the $(\bar r=1)_2$-chart. Then $\overline X_2=\Phi^{P_2}_* X_2$ has $\mu$ as a common factor. It is therefore the desingularized vector-field $\mu^{-1}\overline X_2$, that we study in the following. To perform calculations, we work in three separate charts $(\bar r=1,\bar l_2=1)_{21}$, $(\bar r=1,\bar v=1)_{22}$ and $(\bar r=1,\bar v=-1)_{23}$ with chart-specific coordinates $(\rho_2,v_1,\mu_1)$, $(\rho_2,\mu_2,l_{22})$ and $(\rho_2,\mu_3,l_{23})$ defined by
\begin{align}
(\bar r=1,\bar l_2=1)_{21}&:\quad \begin{cases}
                     v&=\mu_1 v_1,\\
                     l_2&=\mu_1,
                    \end{cases}\eqlab{c21}\\
 (\bar r=1,\bar v=1)_{22}&:\quad \begin{cases}
                     v&=\mu_2,\\
                     l_2&=\mu_2 l_{22},
                    \end{cases}   \eqlab{c22} \\
                    (\bar r=1,\bar v=-1)_{23}&:\quad \begin{cases}
                     v&=-\mu_3,\\
                     l_2&=\mu_3 l_{23}.
                    \end{cases} \eqlab{c23}
\end{align}
See also \figref{blowupVR} for an illustration of these coordinates. 
The change of coordinates between $(\bar r=1,\bar l_2=1)_{21}$ and $(\bar r=1,\bar v=1)_{22}$ is given by 
\begin{align}
 \mu_2=\mu_1 v_1,\quad l_{22}=v_1^{-1},\eqlab{cc2122}
\end{align}
for $\mu_1\ge 0$, $v_1>0$. Similarly, between $(\bar r=1,\bar l_2=1)_{21}$ and $(\bar r=1,\bar v=-1)_{22}$ we have
\begin{align}
 \mu_3=\mu_1 (-v_1),\quad l_{23}=-v_1^{-1},\eqlab{cc2223}
\end{align}
for $v_1<0$. 
We will achieve the desingularization in each of the charts through division of the local vector-fields by $\mu_i$, $i=1,2,3$, respectively.

\begin{remark}\remlab{v1vsrt0}
 Using \eqref{c2} and \eqref{r1veqn}, we have that $v_1$ in \eqref{c21} can be written in terms of $r$ and $\dot r$ as follows:
 \begin{align}
v_1 = \rho_2 \dot r.\eqlab{v1vsrt}
 \end{align}
The equations, we obtain in the $(\bar r=1,\bar l_2=1)_{21}$-chart, see \secref{c21} and \eqref{c21eqns}, are therefore equivalent to \eqref{rttlt}
on $r=\rho_2^2>0$. However, the factor of $\rho_2$ in  \eqref{v1vsrt} induces a compactification of $\dot r$. See also \remref{v1vsrt}.
\end{remark}

\begin{figure}[h!]
 	\begin{center}
 		{\includegraphics[width=.9\textwidth]{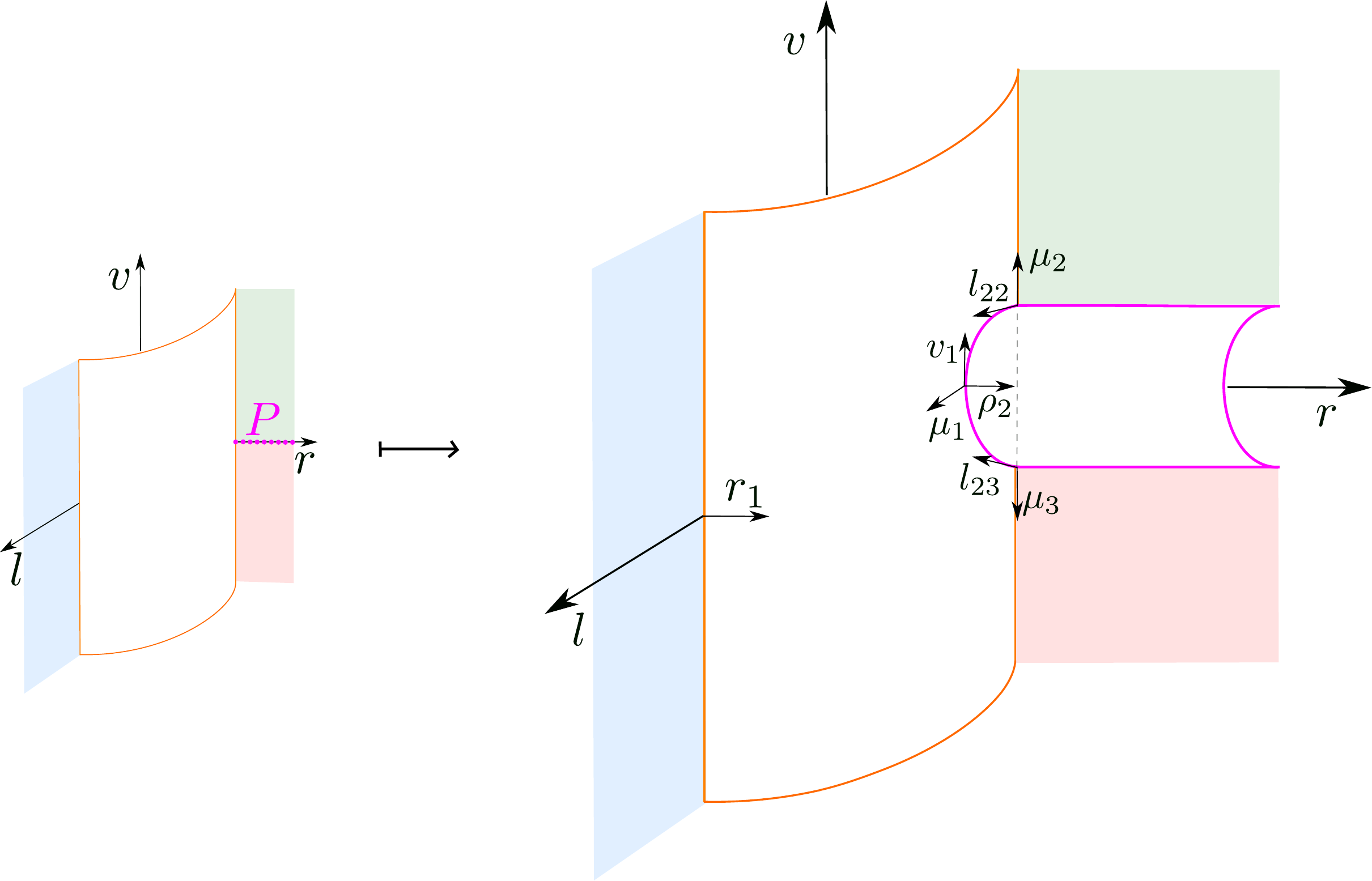}}
 		\caption{Illustration of the cylindrical blowup of the degenerate set $P_2$, see \eqref{phiR2}. }\figlab{blowupVR}
 	\end{center}
 \end{figure}
 
As we will see, orbits in $(\bar r=1,\bar v=1)_{22}$ go unbounded (with $\rho_2\rightarrow \infty$). We will therefore need to compactify the space $(\rho_1,\mu_2,l_{22})$. It turns out that the most convinient way to do this is as follows:
\begin{align}
 (\nu,l_{222})\mapsto \begin{cases}
                             \rho_2 &= \nu^{-1},\\
                             l_{22} &=\nu^3 l_{222},
                            \end{cases}\eqlab{c221}
\end{align}
with $\nu \ge 0,l_{222}\ge 0$, leaving $\mu_2$ fixed.
In this way, $\nu=0$ corresponds to $\rho_2=\infty $ (or $r=\infty$ by \eqref{c1}) and $l_{22}=0$. 
The latter property may seem unnatural, but upon using \eqref{cc2122}, we may realize that it leads to the following  compactification of the $(\rho_2,v_1,\mu_1)$-space associated with the $(\bar r=1,\bar l_2=1)_{21}$-chart:
\begin{align}
 (\nu,v_{11})\mapsto \begin{cases}
                             \rho_2 &= \nu^{-1},\\
                             v_{1} &=\nu^{-3} v_{11},\\
                              \mu_1 &=\nu^3 \mu_{11}
                            \end{cases}\eqlab{cv11}
\end{align}
where $v_{11}=l_{222}^{-1}$, $\mu_{11}=\mu_2 l_{222}$. Using \eqref{cc2223}, we then also obtain the following compactification in the $(\bar r=1,\bar v=-1)_{23}$-chart:
\begin{align}
 (\nu,l_{23})\mapsto \begin{cases}
                             \rho_2 &= \nu^{-1},\\
                             l_{23} &=\nu^{3} l_{233}
                            \end{cases}\eqlab{c233}
\end{align}
with $\nu\ge 0,l_{233}\ge 0$, leaving $\mu_3$ fixed.

We present the final geometric picture of our compactified phase space in a schematic way in \figref{full}. This figure also illustrates the coordinates used at $r=\infty$. 

In the following section, we study the dynamics in each of the charts. 

\begin{remark}
 In fairness, orbits within $\{r=0\}$ are also unbounded in the $v$-direction for $l>0$, see \figref{rvl}, and from this perspective, it is also desirable to compactify the $v$-direction. However, for simplicity we have chosen not to include this. 
\end{remark}

\begin{figure}[h!]
 	\begin{center}
 		{\includegraphics[width=.9\textwidth]{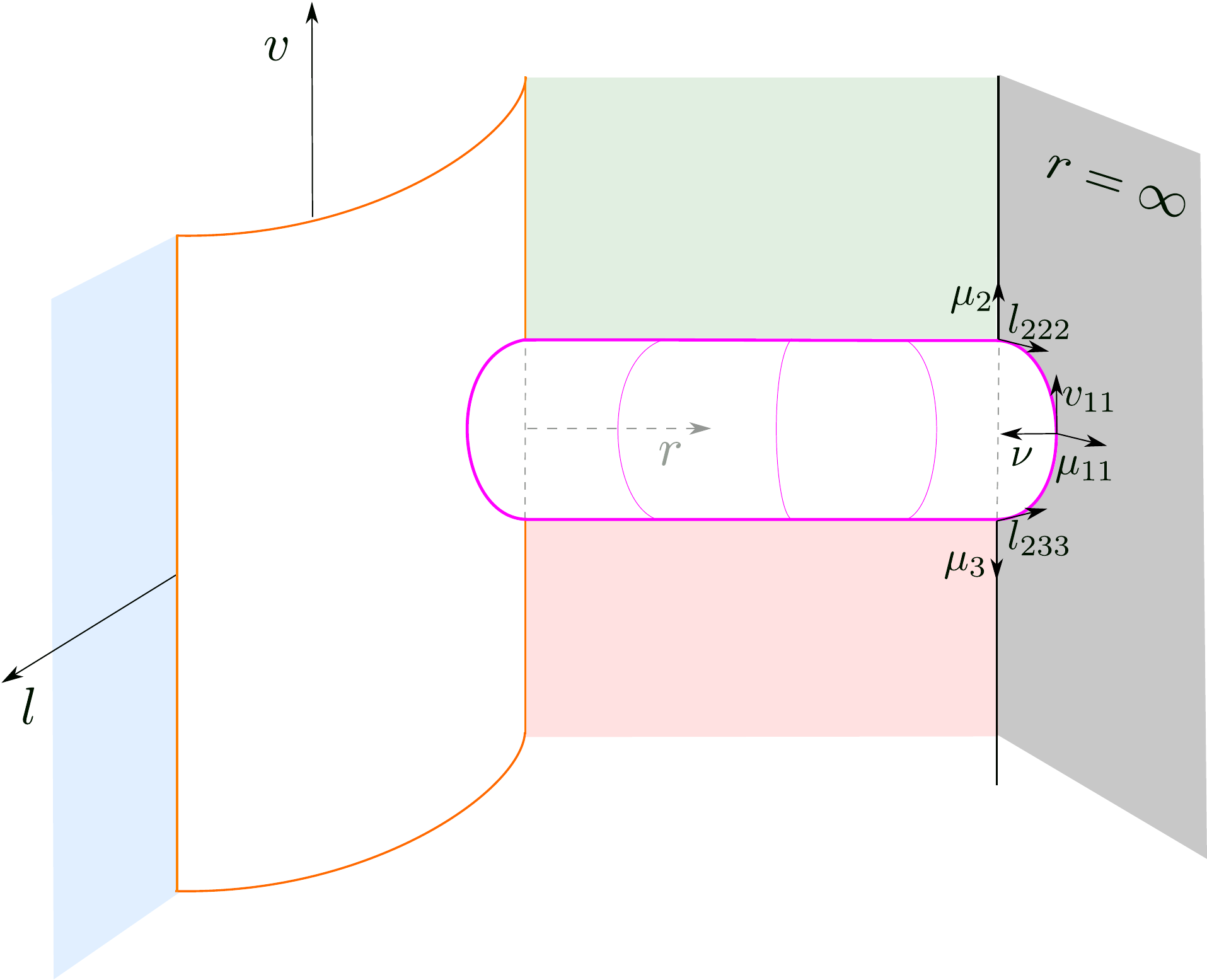}}
 		\caption{Illustration of the full blown up and compactified system. We indicate the coordinates used at $r=\infty$, see also \eqref{c221}, \eqref{cv11},\eqref{c233}. }\figlab{full}
 	\end{center}
 \end{figure}

\subsection{Chart $(\bar l=1)_1$}
By inserting \eqref{c1} into \eqref{rvl2}, we obtain
\begin{equation}\eqlab{l1eqns20}
\begin{aligned}
 r_1' &=vr_1^3 +2\delta r_1^4 \rho_1^3,\\
 v' &= -{r_1+1}-2\delta vr_1^3 \rho_1^3,\\
  \rho_1' &=-\delta r_1^3 \rho_1^4,
\end{aligned}
\end{equation}
after desingularization, corresponding to division of the right hand side by $\rho_1^4$. Setting $\rho_1=l$ and dividing the right hand side by $r_1^3$, we obtain \eqref{l1eqns2}. Consequently, within $\rho_1=0$, we have the Hamiltonian system with Hamiltonian function $H(r_1,v)=\frac12 v^2 +\frac{(r_1-1)^2}{2r_1^2}$, recall \eqref{H0func}, having periodic orbits $\Gamma_1(h)$, $h\in \left(0,\frac12\right)$ within $\rho_1=0$, surrounding the center $(r_1,v)=(1,0)$. \thmref{main1} and \thmref{main2} gave the existence of stable manifolds of $(r_1,v,\rho_1)=(1,0,0)$ and $\Gamma_1(h)$, $h\in \left(0,\frac12\right)$. The orbit $\Gamma_1\!\left(\frac12\right)$, defined by $H(r_1,v)=\frac12$ within $\rho_1=0$, is a separatrix, separating bounded (periodic) orbits from the unbounded ones ($H(r_1,v)\ge \frac12$), see \lemmaref{Ham} and \figref{r1v}.
\subsection{Chart $(\bar r=1)_2$}
By inserting \eqref{c2} into \eqref{rvl2}, we obtain the equations:
\begin{equation}\eqlab{eqnsrho2vl2}
\begin{aligned}
  \rho_2' &= \frac12 \rho_2 v,\\
 v' &=l_2^2(l_2^2-1)-2\delta \rho_2^3 v l_2,\\
 l_2' &=-\frac12 l_2 \left(v+2\delta \rho_2^3 l_2\right),
\end{aligned}
\end{equation}
after desingularization (corresponding to division of the right hand side by $\rho_2^4$).
Here we have two invariant planes defined by $\rho_2=0$ and $l_2=0$.
Within the former, we rediscover the center at $(v,l_2)=(0,1)$ and the periodic orbits $\Gamma_2(h)$, $h\in \left(0,\frac12\right)$, given by 
\begin{align}\eqlab{H2}
 H_2(l_2,v):=H(l_2^{-2},v) =  \frac12 v^2 +\frac12 -l_2^2+\frac12 l_2^4=h,
\end{align}
surrounding the center.
However, $\Gamma_1\!\left(\frac12\right)$ now becomes a bounded orbit $\Gamma_2\!\left(\frac12\right)$, which is homoclinic to the degenerate point $\gamma_2$ defined by $(\rho_2,v,l_2)=(0,0,0)$, see \figref{rho2vl2}. All other points $(0,v,0)$ on the $v$-axis, are partially hyperbolic, the linearization having eigenvalues $\pm \frac12$, $0$. All orbits $H_2(l_2,v)=h$ with $h>\frac12$ are heteroclinic connections within $\rho_2=0$ between points on the $v$-axis. Next, within $l_2=0$ we have 
\begin{align*}
 \rho_2' &= \frac12 \rho_2 v,\\
  v' &=0.
\end{align*}
The $\rho_2$-axis is the set of degenerate equilibria $P_2$, which is blown up by \eqref{phiR2}. Notice that $H_2(0,0)=\frac12 $, and since $H_2$ is independent of $\rho_2$, we conclude using \eqref{Hinf} that 
\begin{lemma}\lemmalab{EinfEq1onP2}
$\vert \mathcal E_\infty\vert=1$ on $P_2$.  
\end{lemma}

\begin{figure}[h!]
 	\begin{center}
 		{\includegraphics[width=.69\textwidth]{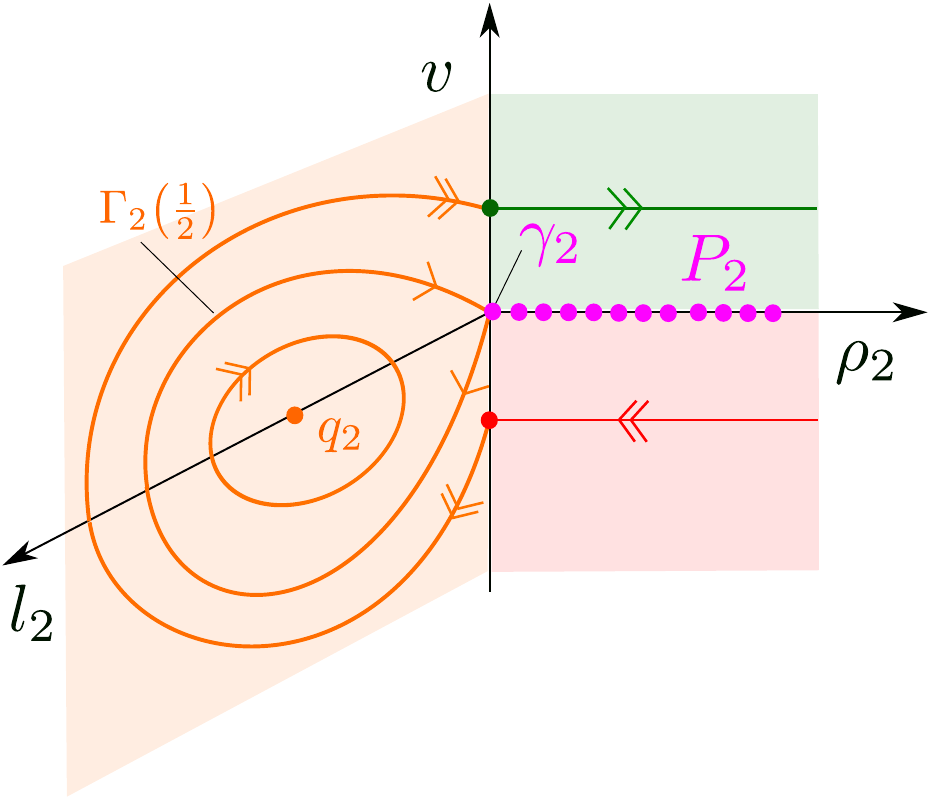}}
 		\caption{Dynamics in the $(\bar r=1)_2$-chart. Here the special orbit $\Gamma_1\!\left(\frac12\right)$ in the $(\bar l=1)_1$-chart, becomes a homoclinic orbit $\Gamma_2\!\left(\frac12\right)$ to a degenerate equilibrium $\gamma_2$.  }\figlab{rho2vl2}
 	\end{center}
 \end{figure}
 
\subsection{Chart $(\bar r=1,\bar l=1)_{21}$}\seclab{c21}
By inserting \eqref{c21} into \eqref{eqnsrho2vl2}, we obtain the following equations
\begin{equation}\eqlab{c21eqns}
\begin{aligned}
 \rho_2' &=\frac12 \rho_2 v_1,\\
 v_1' &=\mu_1^2+\frac12 v_1^2-1-\delta \rho_2^3v_1,\\
 \mu_1'&=-\left(\frac12 v_1+\delta \rho_2^3\right)\mu_1,
\end{aligned}
\end{equation}
after desingularization (corresponding to division of the right hand side by $\mu_1$). 
$\mu_1=0$ corresponds to the blowup of $P_2$, and within this invariant subspace, we have that
\begin{equation}\eqlab{rho2v1eqns}
\begin{aligned}
  \rho_2' &= \frac12 \rho_2 v_1,\\
 v_1' &=\frac12 v_1^2-1-\delta \rho_2^3v_1.
\end{aligned}
\end{equation}
We have two hyperbolic equilibria $\gamma_{21}^\pm$ along $\rho_2=0$ given by $v_1=\pm \sqrt{2}$, the former is an unstable node while the latter is a stable node. As equilibria points $(0,\pm \sqrt{2},0)$ of the full system, they are hyperbolic saddles and the separatrix $\Gamma_1\!\left(\frac12\right)$ from $(\bar r=1)_1$, now denoted by $\Gamma_{21}\!\left(\frac12\right)$ and given by $H_2(\mu_1,\mu_1 v_1)=\frac12$ within $\rho_2=0$, is now a heteroclinic orbit connecting the two hyperbolic saddles. We summarize the findings in \figref{rho2mu1v1}.

\begin{remark} \remlab{v1vsrt}
Following \remref{v1vsrt0}, \eqref{rho2v1eqns} is equivalent to \eqref{rttlt} with $l=0$:
\begin{align*}
 \ddot r = -\delta \dot r -\frac{1}{r^2}.
\end{align*}
In particular, setting $v_1=\rho_2 u$ with $u=\dot r$ and $\rho_2=\sqrt{r}$ transforms \eqref{rho2v1eqns} into
\begin{align*}
  r' &=ur^2,\\
 u'&=-\delta u r^2 -1,
\end{align*}
which is studied in \cite[Proposition 3.1]{margheri2017a}. The advantage of working with \eqref{rho2v1eqns} is that the $u$-axis becomes compactified. In particular, the heteroclinic orbits, connecting $\gamma_{21}^\pm$ within $\mu_1=0$, see \figref{rho2mu1v1}, (called ejection-collision orbits in \cite{margheri2014a}) become unbounded in the $(r,u=\dot r)$-coordinates. 
\end{remark}

\begin{figure}[h!]
 	\begin{center}
 		{\includegraphics[width=.69\textwidth]{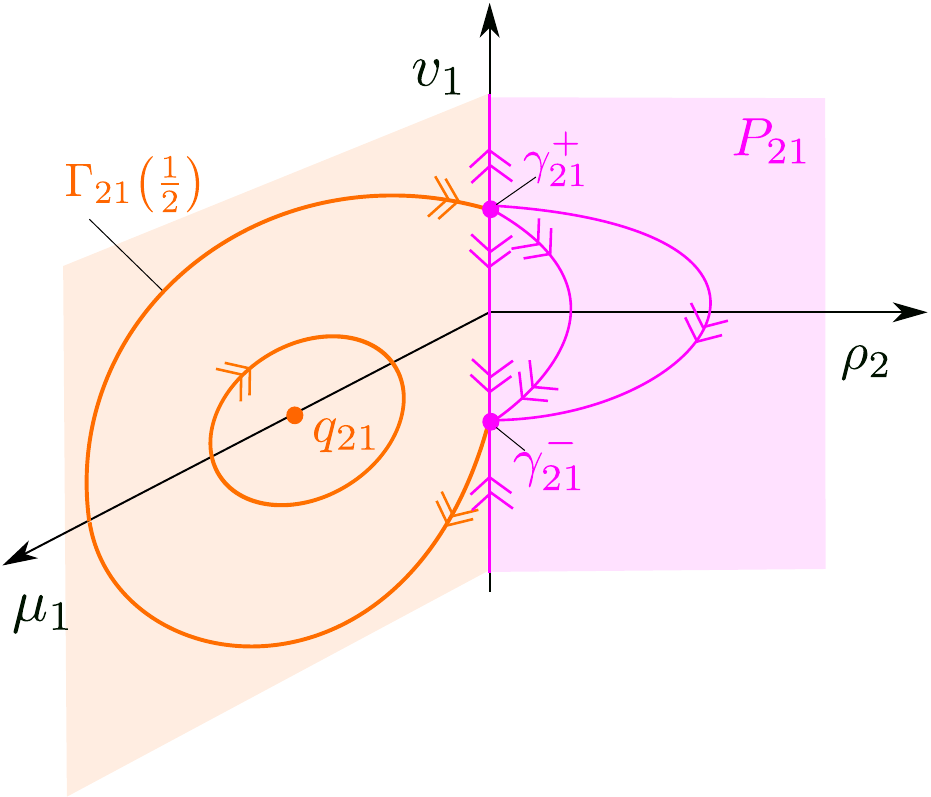}}
 		\caption{Dynamics in the $(\bar r=1,\bar l=1)_{21}$-chart. Here the special orbit $\Gamma_1\!\left(\frac12\right)$ in the $(\bar l=1)_1$-chart, becomes a heteroclinic orbit $\Gamma_{21}\!\left(\frac12\right)$, connecting two hyperbolic saddles $\gamma_{21}^\pm$.}\figlab{rho2mu1v1}
 	\end{center}
 \end{figure}
\subsection{Chart $(\bar r=1,\bar l=1)_{22}$}\seclab{r1l122}
By inserting \eqref{c22} into \eqref{eqnsrho2vl2}, we obtain the following equations
\begin{equation}\eqlab{c22eqns}
\begin{aligned}
 \rho_2' &=\frac12 \rho_2,\\
  \mu_2' &=l_{22} \mu_2\left(l_{22}^3 \mu_2^2-2\delta \rho_2^3 - l_{22}\right),\\
 l_{22}'&=-l_{22}\left(l_{22}^4 \mu_2^2-\delta l_{22}\rho_2^3 -l_{22}^2 + \frac12\right),
\end{aligned}
\end{equation}
after desingularization (corresponding to division of the right hand side by $\mu_2$).
All of the three invariant planes defined by $\mu_2=0$, $\rho_2=0$ and $l_{22}=0$, respectively, are invariant. Within $\mu_2=0$, we find
\begin{equation}\eqlab{mu2eq0}
\begin{aligned}
 \rho_2' &=\frac12 \rho_2,\\
  l_{22}'&=-l_{22}\left(-\delta l_{22}\rho_2^3 -l_{22}^2 + \frac12\right).
\end{aligned}
\end{equation}
Here we rediscover $\gamma_{21}^+$ from the $(\bar r=1,\bar l=1)_{21}$-chart as a hyperbolic unstable node within $\mu_2=0$ given by
\begin{align*}
 \gamma_{22}^+ :\quad (\rho_2,\mu_2,l_{22})=(0,0,1/\sqrt{2}),
\end{align*}
see also \eqref{cc2122}. At the same time, $(\rho_2,l_{22})=(0,0)$ is a hyperbolic saddle for \eqref{mu2eq0}, the linearization having eigenvalues $\pm \frac12$. The two axes, $\rho_2$ and $l_{22}$, are the associated unstable and stable manifolds, respectively. 

Next, within $l_{22}=0$, we have 
\begin{align*}
 \rho_2' &=\frac12 \rho_2,\\
 \mu_2' &=0,
\end{align*}
and the $\mu_2$-axis is therefore a line of saddle points of \eqref{c22eqns}, having $l_{22}=0$ ($\rho_2=0$) as its unstable manifold (stable manifold, respectively). We summarize the findings in \figref{rho2mu2l22}.
\begin{figure}[h!]
 	\begin{center}
 		{\includegraphics[width=.69\textwidth]{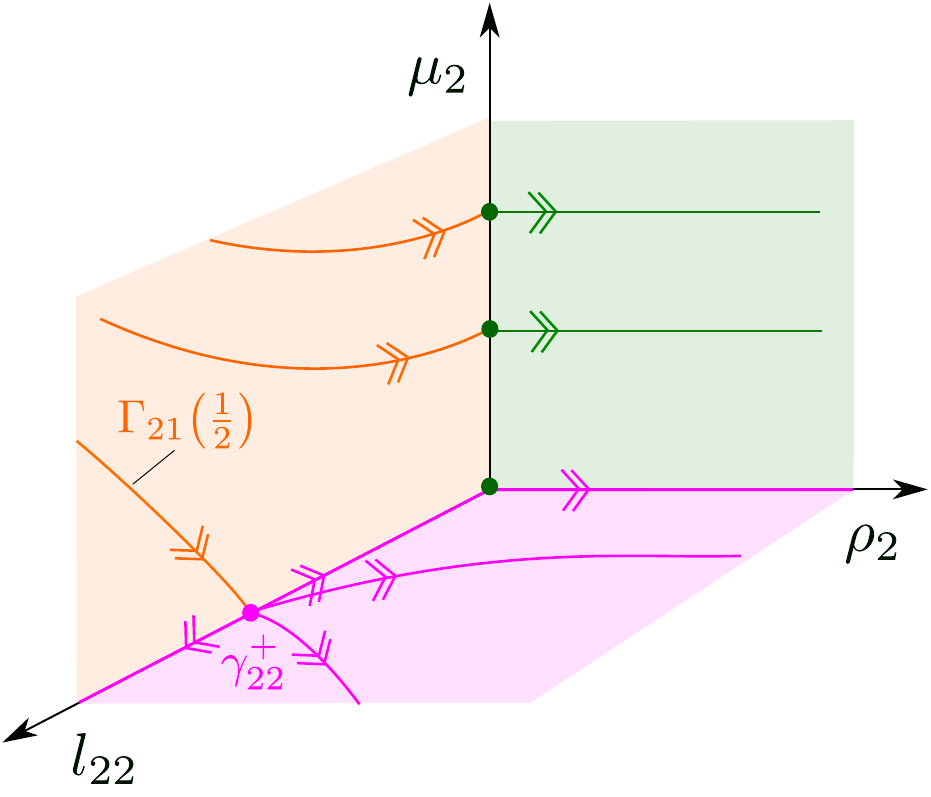}}
 		\caption{Dynamics in the $(\bar r=1,\bar l=1)_{22}$-chart, compare with  \figref{rho2mu1v1}.}\figlab{rho2mu2l22}
 	\end{center}
 \end{figure}
\subsection{Chart $(\bar r=1,\bar l=1)_{23}$}
By inserting \eqref{c23} into \eqref{eqnsrho2vl2}, we obtain the following equations
\begin{equation}\eqlab{c23eqns}
\begin{aligned}
 \rho_2' &=\frac12 \rho_2,\\
  \mu_3' &=l_{23} \mu_3\left(l_{23}^3 \mu_3^2-2\delta \rho_2^3 - l_{23}\right),\\
  l_{23}'&=-l_{23}\left(l_{23}^4 \mu_3^2-\delta l_{23}\rho_2^3 -l_{23}^2 + \frac12\right),
\end{aligned}
\end{equation}
after desingularization (corresponding to division of the right hand side by $\mu_3$).
The analysis in this chart is almost identical to the analysis in $(\bar r=1,\bar l=1)_{22}$. In particular, we rediscover the equilibrium point $\gamma^-_{21}$ from the $(\bar r=1,\bar l=1)_{21}$-chart, now given by
\begin{align*}
 \gamma_{23}^-:\quad (\rho_2,\mu_3,l_{23})=(0,0,1/\sqrt{2}).
\end{align*}
At the same time, we have that the $\mu_3$-axis is line of saddle points of \eqref{c23eqns}, having $l_{22}=0$ ($\rho_2=0$) as its stable manifold (unstable manifold, respectively). An illustration of the dynamics in this chart can be obtained by taking the time-reversal of the diagram \figref{rho2mu2l22} (replacing (a) $l_{22}$, $\mu_2$ by $l_{23}$ and $\mu_3$, respectively,  (b) $\gamma_{22}^+$ by $\gamma_{23}^-$ and finally (c) replacing (in line with \figref{full}) the color green of the invariant plane $l_{22}=0$ to red).

\subsection{Compactification in the $(\bar r=1,\bar l=1)_{21}$-chart}\seclab{inf1}
Upon applying the transformation of $(\rho_2,v_1)$ defined by \eqref{cv11} to the system \eqref{c21eqns}, we obtain the following equations
\begin{equation}\eqlab{c21infeqns}
\begin{aligned}
  \dot \nu &= - \frac12 \nu v_{11},\\
  \dot v_{11} &= -v_{11}(\delta+v_{11})+\nu^6 (\mu_{11}^2\nu^6-1),\\
  \dot \mu_{11} &= \mu_{11} (v_{11}-\delta),
\end{aligned}
\end{equation}
after desingularization (corresponding to multiplication of the right hand side by $\nu^3$). Here $\nu=0$, which corresponds to $r=\infty$, is an invariant set, upon which we find the following:
\begin{equation}\eqlab{nueq0}
\begin{aligned}
 \dot v_{11} &= -v_{11}(\delta+v_{11}),\\
 \dot \mu_{11} &= \mu_{11} (v_{11}-\delta).
 \end{aligned}
 \end{equation}
 We find two equilibria: $p_{21}^+:\,(\nu,v_{11},\mu_{12})=(0,0,0)$ and $p_{21}^-:\,(\nu,v_{11},\mu_{12})=(0,-\delta,0)$ of \eqref{c21infeqns}, both of which are hyperbolic for the reduced system \eqref{nueq0} within $\nu=0$. Indeed, the linearization of \eqref{nueq0} around $p_{21}^+$ produces the eigenvalues  $-\delta,-\delta$ (semi-simple), whereas the linearization of \eqref{nueq0} around $p_{21}^-$ produces $\delta,-2 \delta$ as eigenvalues. Consequently, $p_{21}^+$ is a stable node for \eqref{nueq0}, whereas $p_{21}^-$ is a saddle. 
  
  $\mu_{11}=0$ is also an invariant set for \eqref{c21infeqns}, upon which we find the following:
  \begin{equation}\eqlab{mu1eq0}
  \begin{aligned}
   \dot \nu &= - \frac12 \nu v_{11},\\
 \dot v_{11} &= -v_{11}(\delta+v_{11})-\nu^6.
  \end{aligned}
  \end{equation}
  While $p_{21}^-$ is clearly an unstable node for these equations,
$p_{21}^+$ is only semi-hyperbolic, the linearization having eigenvalues $0$ and $-\delta$. It is a simple calculation, to show that the associated center manifold $W^c(p_{21}^+)$ takes the following graph form:
\begin{align}
 W^c(p_{21}^+):\quad v_{11} = -\frac{1}{\delta} \nu^6(1+\mathcal O(\nu^6)).\eqlab{Wc}
\end{align}
The $\mathcal O(\nu^6)$-term is a smooth function of $\nu^6$ (and $\delta$). This follows from the fact that the first equation of \eqref{mu1eq0} can be written as $(\nu^6)' = -3 \nu^6 v_{11}$. Inserting \eqref{Wc} into \eqref{mu1eq0}
gives
\begin{align*}
 \dot \nu &=\frac{1}{2\delta}\nu^7(1+\mathcal O(\nu^6)),
\end{align*}
and $\nu>0$ is therefore increasing on $W^c(p_{21}^+)$.
The center manifold is therefore unique as the unstable set of $p_{21}^+$ and 
$p_{21}^+$ is a nonhyperbolic saddle for \eqref{mu1eq0}.  
  We illustrate our findings in \figref{numu2v11}.
\begin{remark}\remlab{Wc}
It is a simple calculation to show that the time used in \eqref{c21infeqns} coincides with the original time in \eqref{rvl}, which is why we use $\dot{()}$ instead of $()'$, recall \remref{time}. 

Using \eqref{c2}, \eqref{v1vsrt}, and \eqref{cv11}, we can write \eqref{Wc} in the $(r,\dot r)$-plane as a graph
\begin{align*}
 \dot r &=-\frac{1}{\delta r^2}(1+\mathcal O(r^{-3})),
\end{align*}
over $r\gg 1$. We see that $\dot r\rightarrow 0$ on $W^c$ as $t\rightarrow -\infty$, in line with the results of \cite[Proposition 3.1]{margheri2014a}.
\end{remark}

\begin{figure}[h!]
 	\begin{center}
 		{\includegraphics[width=.69\textwidth]{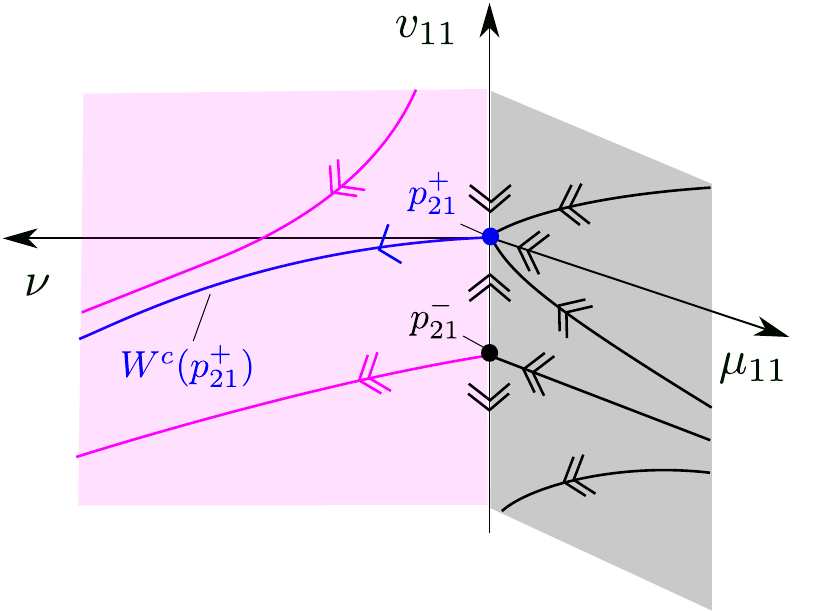}}
 		\caption{Dynamics of \eqref{c21infeqns}. The point $p_{21}^-$ is fully hyperbolic, whereas $p_{21}^+$ is only semi-hyperbolic (we use single-headed arrows to separate center directions from hyperbolic directions (double-headed arrows)). The center manifold $W^c(p_{21}^+)$ of $p_{21}^+$ is unique as the unstable set of $p_{21}^+$. }\figlab{numu2v11}
 	\end{center}
 \end{figure}
\subsection{Compactification in the $(\bar r=1,\bar l=1)_{22}$-chart}\seclab{inf2}
Upon applying the transformation of $(\rho_2,l_{22})$ defined by \eqref{c221} to the system \eqref{c22eqns}, we obtain the following equations
\begin{equation}\eqlab{c22infeqns}
\begin{aligned}
  \nu' &= - \frac12 \nu,\\
  \mu_{2}' &= -\mu_2 l_{222}\left(2\delta+\nu^6 l_{222}-\nu^{12} l_{222}^3 \mu_2^2\right),\\
 l_{222}' &=l_{222}\left(1+\delta l_{222}+\nu^6 l_{222}^2-\nu^{12} l_{222}^4 \mu_2^2\right).
\end{aligned}
\end{equation}
Here $l_{222}=0$ defines an invariant set upon which we have $\nu' = -\frac12 \nu$ and $\mu_2'=0$. Since $l_{222}=0$ corresponds to $l_{22}=0$, see \eqref{c221}, these findings are obviously in agreement with the results in the $(\bar r=1,\bar l=1)_{22}$-chart whenever $\nu>0$, see \secref{r1l122} and \figref{rho2mu2l22} (green plane). On the other hand, $\nu=0$, corresponding to $r=\infty$, is now an invariant set of \eqref{c22infeqns}. In fact, $\nu=l_{222}=0$ is a line of saddle points; the linearization of \eqref{c22infeqns} about any point in this set having eigenvalues $-\frac12,0,1$. The stable manifold of $\nu=l_{222}=0$ is the $(\nu,\mu_2)$-plane, whereas the associated unstable set is the $(\mu_2,l_{222})$-plane. Setting $\nu=0$ in \eqref{c22infeqns} gives
\begin{align*}
  \mu_{2}' &= -2\delta \mu_2 l_{222},\\
  l_{222}' &=l_{222}\left(1+\delta l_{222}\right),
\end{align*}
which has no equilibria within the first quadrant. In particular, $l_{222}>0$ ($\mu_2$) is monotonically increasing (decreasing, respectively). 
Finally, within $\mu_2=0$ we have 
\begin{align*}
 \nu' &= - \frac12 \nu,\\
  l_{222}' &=l_{222}\left(1+\delta l_{222}+\nu^6 l_{222}^2\right),
 \end{align*}
 having a hyperbolic saddle at the origin. We summarize the local findings in \figref{numu2l222}.

 \begin{figure}[h!]
 	\begin{center}
 		{\includegraphics[width=.69\textwidth]{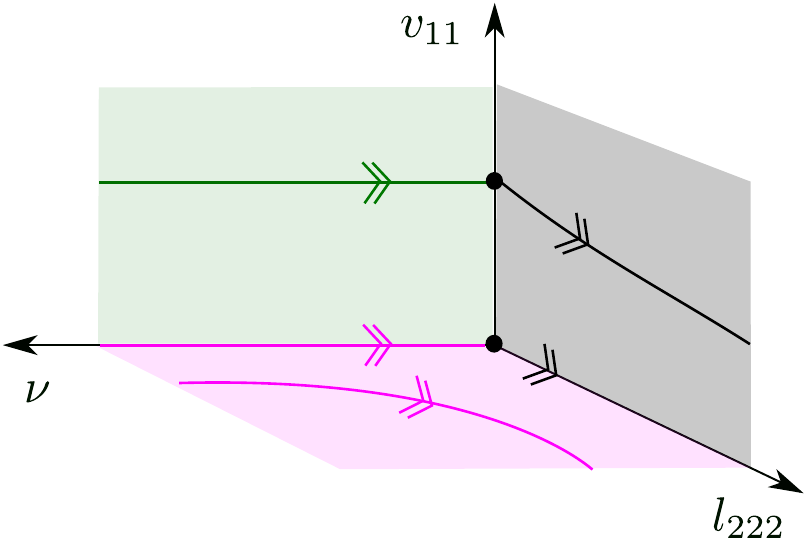}}
 		\caption{Dynamics of \eqref{c22infeqns}. The $v_{11}$-axis is a line of saddle-points. There are no other equilibria in this chart. }\figlab{numu2l222}
 	\end{center}
 \end{figure}
\subsection{Compactification in the $(\bar r=1,\bar l=1)_{23}$-chart}
Upon applying the transformation of $(\rho_2,l_{22})$ defined by \eqref{c221} to the system \eqref{c22eqns}, we obtain the following equations
\begin{equation}\eqlab{c23infeqns}
\begin{aligned}
 \nu' &= \frac12 \nu,\\
  \mu'_{3} &= \mu_3 l_{233}\left(2\delta-\nu^6 l_{233}-\nu^{12} l_{233}^3 \mu_3^2\right),\\
 l_{233}' &=l_{233}\left(-1+\delta l_{233}-\nu^6 l_{233}^2+\nu^{12} l_{233}^4 \mu_3^2\right).
\end{aligned}
\end{equation}
The analysis of these equations is almost identical to the analysis performed in \secref{inf1} and \secref{inf2}. We therefore only present a diagram, see \figref{numu2l233}.

Upon combining all of our findings in the local charts, we obtain the global perspective in \figref{final} (using the geometric viewpoint in \figref{full}). Notice that from this perspective it follows directly that any $\Gamma(h)$ with $h>\frac12$ cannot be the $\omega$-limit set. We elaborate further upon this in the following section. 
\begin{figure}[h!]
 	\begin{center}
 		{\includegraphics[width=.69\textwidth]{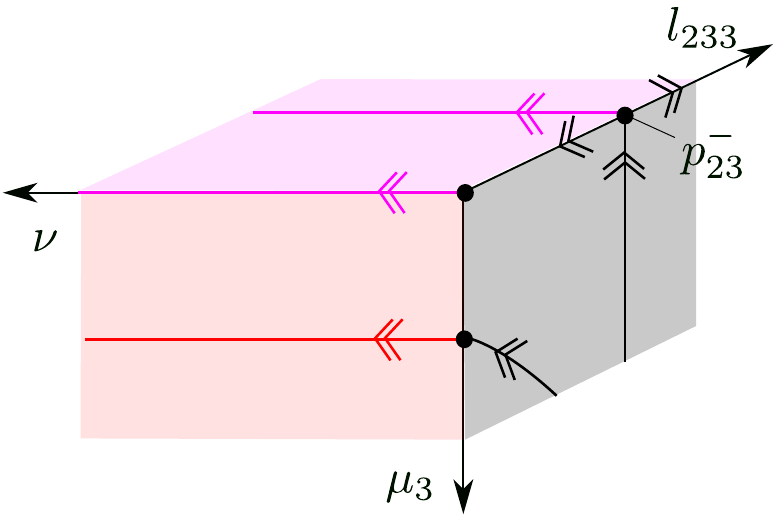}}
 		\caption{Dynamics of \eqref{c23infeqns}. The hyperbolic point $p_{23}^-$ corresponds to the point $p_{22}^-$, see \figref{numu2v11}, upon the coordinate transformation defined by $l_{233}=-v_{11}^{-1}$. The $\mu_3$-axis is a line of saddle-points.}\figlab{numu2l233}
 	\end{center}
 \end{figure}
\section{Discussion}\seclab{disc}
In this paper, we have revisited the linearly damped Kepler problem \eqref{keplerd} with the main purpose of describing the smoothness of the invariant manifolds obtained in \cite{margheri2017a}, see \thmref{main2} and \thmref{main3}. In the process, we identified a separate invariant manifold $W^s(q)$, see \thmref{main1}. This one-dimensional manifold of \eqref{rttlt}, corresponding to orbits becoming more circular as $t\rightarrow \infty$, acts as the center of oscillations along the invariant manifolds of \thmref{main2} and \thmref{main3}, see \figref{final}. 

Finally, in \secref{blowup} we performed a blowup analysis that led to the geometric description of the dynamics illustrated in \figref{final}. In future work, we aim to use a similar approach to study \eqref{keplerd} with the general $(u,\dot u)$-dependent nonlinear damping:
\begin{align*}
\delta = \frac{k \vert \dot u\vert^\alpha}{\vert u\vert^\beta},
\end{align*}
for $\beta,\alpha \in \mathbb N_0$ and $k>0$. This family was considered by Poincar\'e, and he argued formally, see e.g. \cite{margheri2020a}, that for $\alpha$ and $\beta$ sufficiently large, orbits tend to ``circularize'', in the sense that $\mathcal E\rightarrow 0$ for $t\rightarrow \infty$ on an open set. To the best of our knowledge, this remains a conjecture to this today. We hope to solve this conjecture and determine all $\alpha$ and $\beta$ for which ``circularization'' occurs using a similar blowup approach. We emphasize that it is known that circularization does not occur for linear damping (indeed, it is exceptional occording to \thmref{main1} and \thmref{main2}) nor does it occur for the Poynting-Plummer-Damby damping ($\alpha=0$, $\beta=2$). See \cite{diacu1999a} and \cite{margheri2020a} for a general analysis of the case $\alpha=0$, $\beta>0$.

 In the following, we will use the geometric framework of the blowup approach to shed light on $\vert \mathcal E_\infty\vert\rightarrow 1$, recall \lemmaref{Hhge12}.

Firstly, in \figref{simulation} we have used Matlab's ODE45 to simulate the system for initial conditions near the set $\Gamma_2(0.8)$ defined by $\rho_2=0$, $H_2(l_2,v)= 0.8>\frac12$, recall \eqref{H2}. More specifically, in (a) we use the coordinates $(\rho_2,v,l_2)$ of the $(\bar r=1)_2$-chart, see \eqref{c2}, and select $50$ initial conditions (indicated by the cyan cylinder near $v=0$) with $(l_2,v,\rho_2)=(l_{20},0,\rho_{20})$, $H_2(l_{20},0)=0.8>\frac12$, and $\rho_{20}$ in the interval $(0.15,0.25)$ (equispaced). The curve $\Gamma_2(0.8)$ is shown in red and the forward orbits (also in cyan) follow this red curve, but only up until a vicinity  of the $v$-axis (due to the saddle-structure along this set, see \figref{rho2vl2} and the discussion of \figref{final2} below). From here $\rho_2$ increases and the orbits eventually contract towards the degenerate set $P_2$. In (b), we use the coordinates $(\rho_2,v_1,\mu_1)$ of the $(\bar r=1,\bar l_2=1)_{21}$-chart, where $P_2$ has been blown up, to illustrate the same (computed) orbits. Here we see that the motion along $P_2$ in (a) is due to the contraction towards the center manifold $W^c(p_{21}^+)$; this is also shown in (c) using a projection onto the $(\nu,v_{11})$-plane, recall \eqref{cv11}. Beyond the motion along $W^c(p^+)$ in \figref{simulation} (b), the cyan curves come close together near $\gamma_{21}^-$ and follow the heteroclinic orbit $\Gamma_{21}\!\left(\frac12\right)$, connecting $\gamma_{21}^-$ and $\gamma_{21}^+$. Once passing close to $\gamma_{21}^+$ the cyan curves move close to $\mu_1=0$ plane again before returning to $\gamma_{21}^-$ and $\Gamma_{21}\!\left(\frac12\right)$. This process repeats itself. 
The computations were performed with low tolerances ($10^{-12}$) and ended when the value of $l$ reached $10^{-4}$. 

The results in \figref{simulation} are in agreement with the findings in \secref{blowup} (see \figref{final}). We emphasize this point further in \figref{final2}, where we illustrate a ``singular'' orbit of a point on the orange cylinder. Here singular refers to the fact that it lies on the blowup space (it is therefore not a true orbit of the system \eqref{keplerd} since it occurs on $r=0$) and at the same time we use a concatenation of unstable and stable manifolds across saddle points. The forward orbit of initial conditions starting near the cyan curve on the orange cylinder (with $0<l\ll 1$) (as in \figref{simulation} in the local charts $(\bar r=1)_2$ and $(\bar r=1,\bar l_2=1)_{21}$) will track this curve (meaning that the orbit remains close to local copies within compact subsets of the local charts) up until $\gamma^+$. This follows from \secref{blowup}, but we will leave out further details. 

At $\gamma^+$ there is no unique forward (singular) orbit, since $\gamma^+$ is an unstable node on $\mu=0$ (the $l=0$ cylinder corresponding to the blowup of $P_2$, see \eqref{phiR2}), see also \figref{rho2mu1v1}. Nevertheless, since (a) $\gamma^+$ connects to $\gamma_-$, (b) $\gamma_-$ is contracting on $\rho=0$, and (c) $\Gamma\!\left(\frac12\right)$ connects back to $\gamma^+$ this process repeats itself for an actual orbit, starting near the cyan disc with $0<l\ll 1$. However, since $l$ is monotonically decreasing, the excursions from $\gamma^+$ to $\gamma^-$ will get closer and closer to the purple connection (along the half-circle $(\bar v,\bar l_2)\in S^1$ $\bar l_2\ge 0$ within $\rho_2=0$). We observe this numerically (also indicated in \figref{simulation} (b) by the blue segments close to $\mu_1=0$). 

 However, and this is the main point, since $\Gamma_1\!\left(\frac12\right)$ cannot be the $\omega$-limit set, see \lemmaref{Hhge12}, and since $l=0$ is invariant, the $\omega$-limit set of the initial conditions in \figref{simulation} (or near the cyan disc in \figref{final2}) has to be $\Gamma_1(h)$ with $h<\frac12$ but $h\sim \frac12$. In other words, these initial conditions belong to a $W^s(\Gamma(h))$ for some $h\in (0,\frac12)$, where $h\rightarrow \frac12$ as $l_0\rightarrow 0$. The implication of this is that the cylinders $W^s(\Gamma(h))$ have a dramatic limit $h\rightarrow \frac12$ (corresponding to $\vert \mathcal E_\infty\vert\rightarrow 1^-$ see \eqref{Hinf}). Indeed, the two-dimensional cylindrical manifolds $W^s(\Gamma(h))$ with $h<\frac12$ but $h\sim \frac12$ open up like a flower, stretching (and collapsing for $h\rightarrow \frac12$) both along the orange and the purple cylinder in \figref{final2} (corresponding to $r$ very small and very large, respectively). 

 At the same time, $l=0$ defines an invariant set (of collinear orbits) and the $\omega$-limit set of any point within this set is $\gamma^-$. Orbits having $\gamma^+$ as the $\alpha$-limit set (within $l=0$) are called ejection-collision orbits, whereas orbits having either $p^-$ or $p^+$ as the $\alpha$-limit set (going along $W^s(p_{21}^+)$ in the latter case) are referred to as capture-collision orbits, see e.g. \cite{diacu1999a,margheri2014a}. Interestingly, the orbits in \figref{simulation} with $0<l\ll 1$ exhibit both types of behaviour, first capture-collision by following close to $W^c(p_{21}^+)$ and then subsequently (and repeatedly) ejection-collision type behaviour by passing close to $\gamma^+$ and $\gamma^-$. We illustrate this in our final figure \figref{Rvst} plotting $r$ as a function of $t$ along one of the cyan curves in \figref{simulation}(a).

\begin{figure}[h!]
 	\begin{center}
 		\subfigure[$(\bar r=1)_2$-chart]{\includegraphics[width=.465\textwidth]{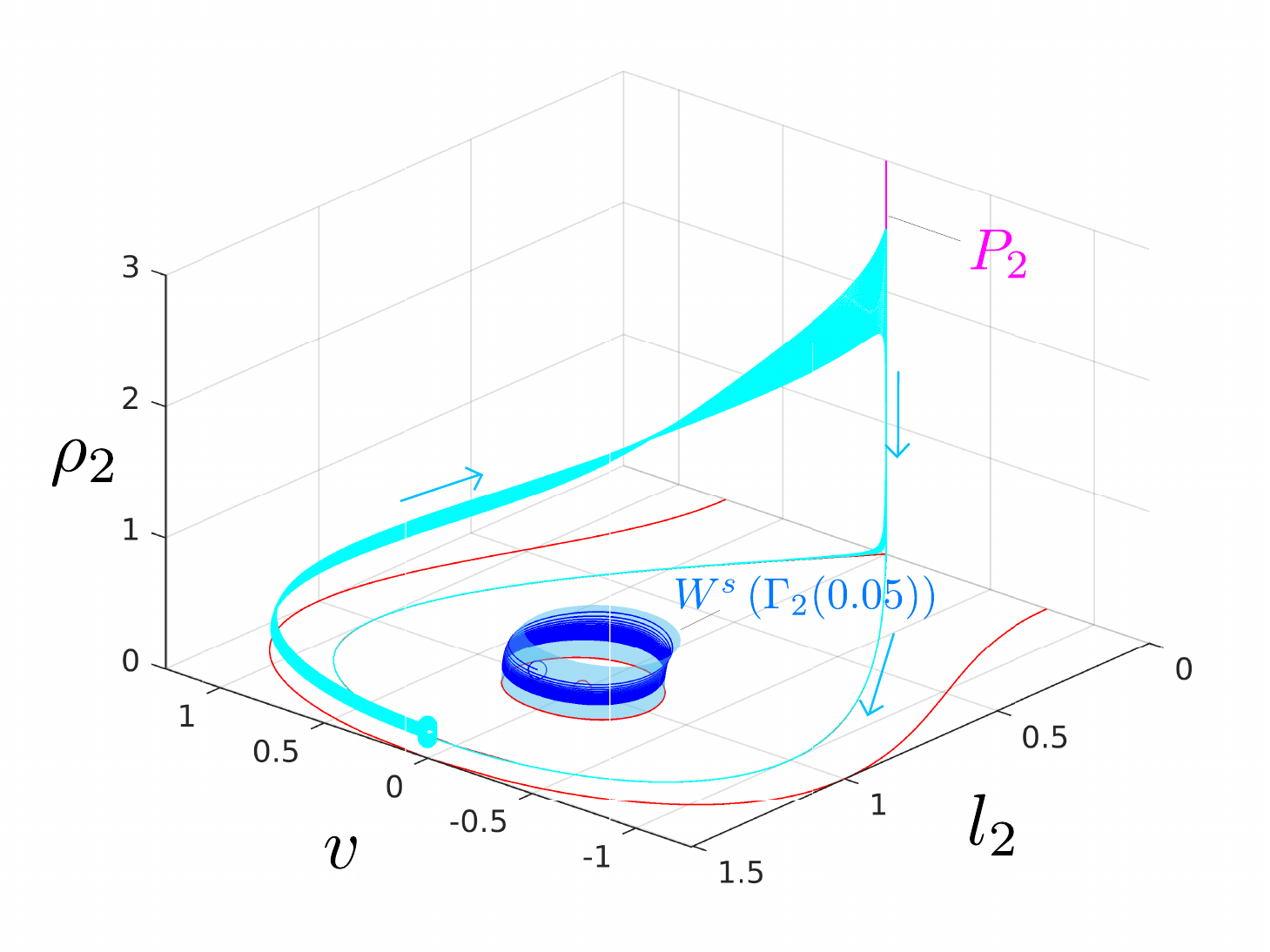}}
 		\subfigure[$(\bar r=1,\bar l_2=1)_{21}$-chart]{\includegraphics[width=.465\textwidth]{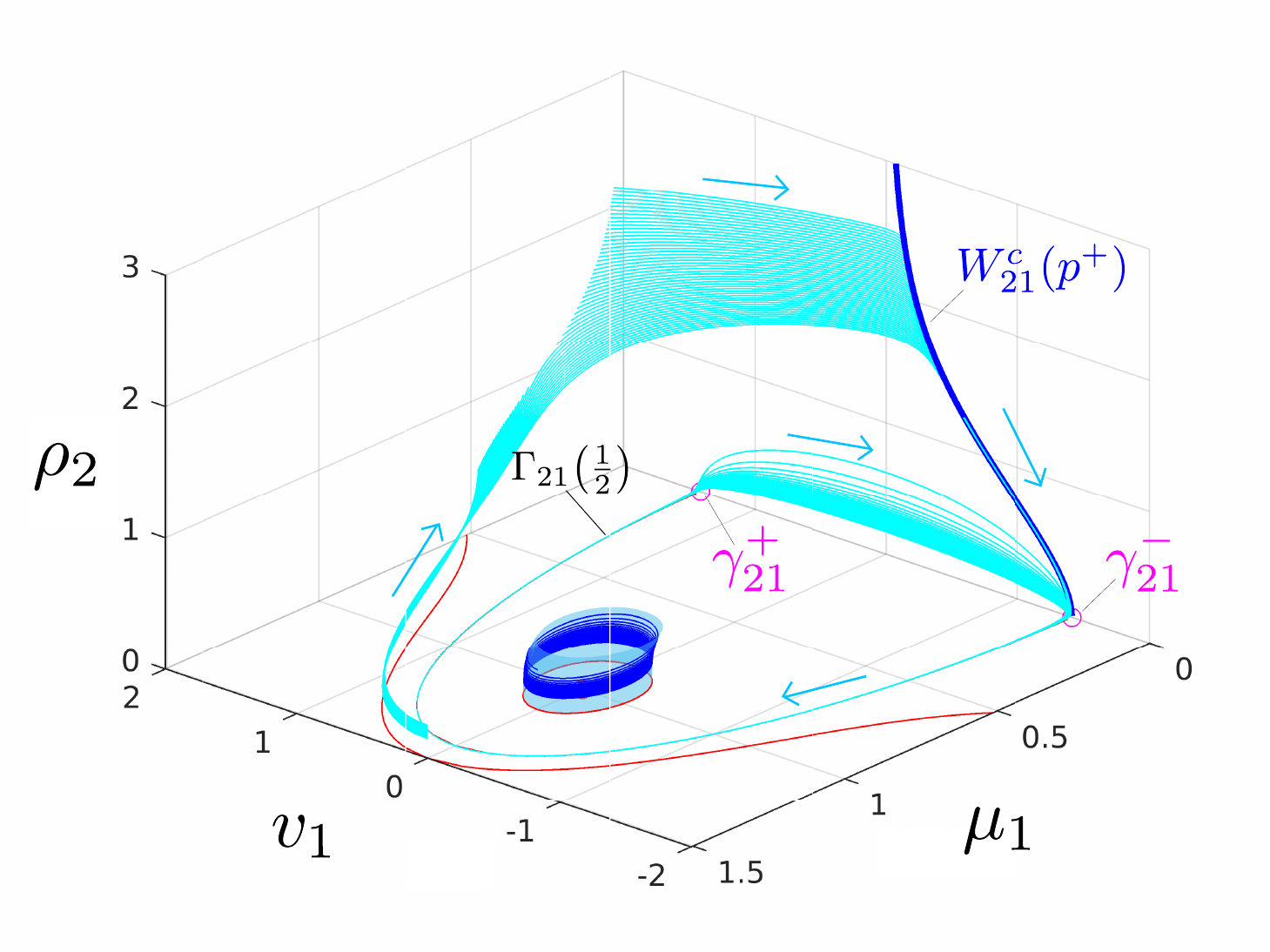}}
 		\subfigure[]{\includegraphics[width=.465\textwidth]{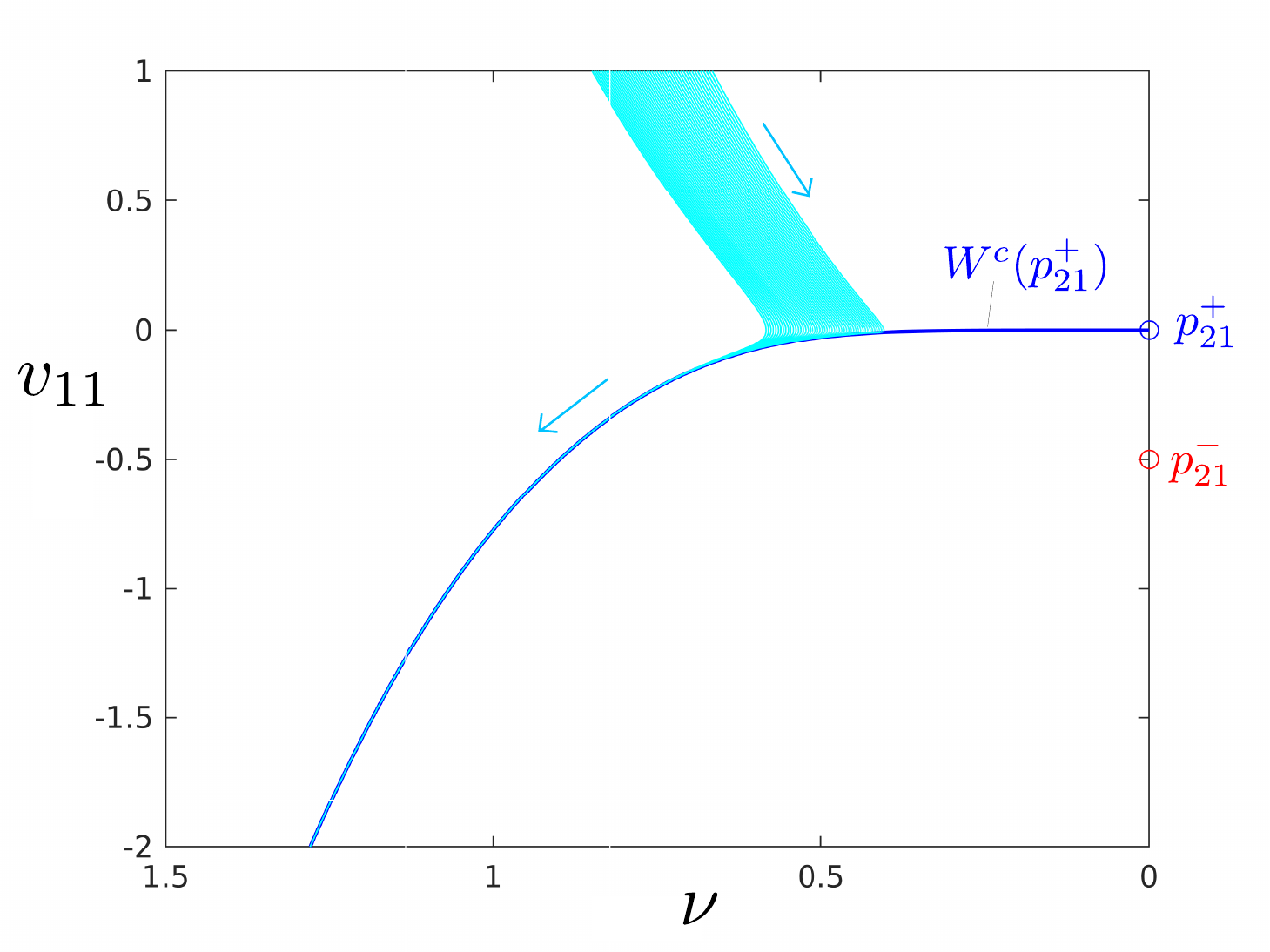}}
 		\caption{Numerical forward integration (using Matlab's ODE45 with low tolerances) of initial conditions (cyan cylinder near $v=0$) starting close to a $h=0.8$ level set of the Hamiltonian function $H_2$ within $\rho_2=0$ (red curves in (a) and (b)). The corresponding orbits (also in cyan) follow an itinerary that can be explained by our blowup analysis (see text and compare with  \figref{final2} below). In (a) we use the $(\rho_2,v,l_2)$-coordinates of the $(\bar r=1)_2$-chart, whereas in (b) we use the $(\rho_2,v_1,\mu_1)$-coordinates of the $(\bar r=1,\bar l_2=1)_{21}$-chart. In (c) we use a projection onto the $(\nu,v_{11})$-plane, see \eqref{cv11}. Here we see that the cyan orbits all follow the center manifold $W^c(p_{21}^+)$ which lead these orbits towards $\gamma_{21}^-$ and $\Gamma_{21}\!\left(\frac12\right)$, see (b).}\figlab{simulation}
 	\end{center}
 \end{figure}

\begin{figure}[h!]
 	\begin{center}
 		{\includegraphics[width=.9\textwidth]{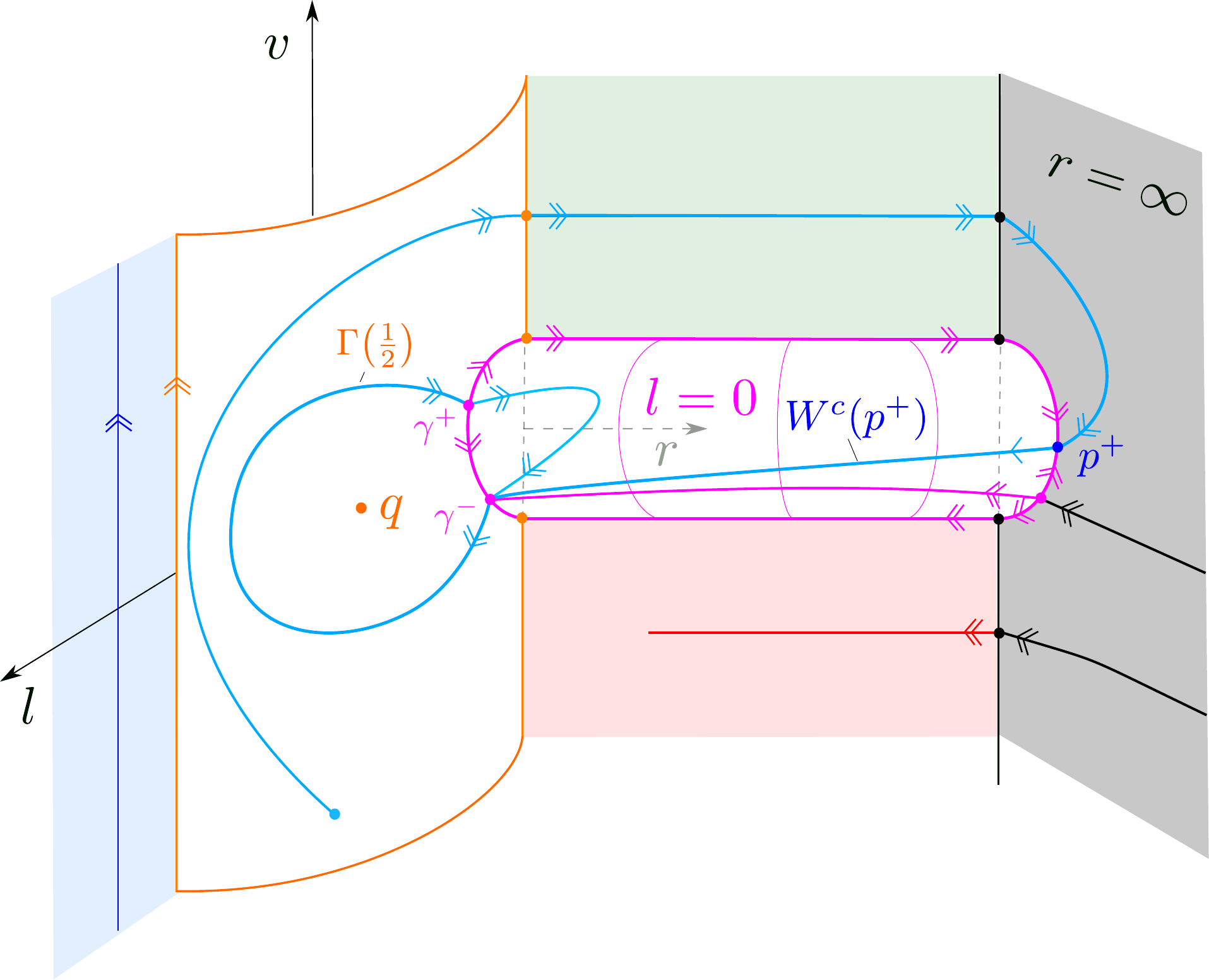}}
 		\caption{The forward flow of an initial condition (cyan disc) with sufficiently small angular momentum $0<l \ll 1$ and near a $h>\frac12$ level set of the Hamiltonian function $H$, will follow the ``singular orbit'' in cyan  (i.e. the forward orbit will remain $o(1)$-close to local copies of this cyan curve in the local charts as the initial value of the angular momentum $l\rightarrow 0^+$), at least up until $\gamma^+$, see \figref{simulation}.}\figlab{final2}
 	\end{center}
 \end{figure}

 \begin{figure}[h!]
 	\begin{center}
 		{\includegraphics[width=.65\textwidth]{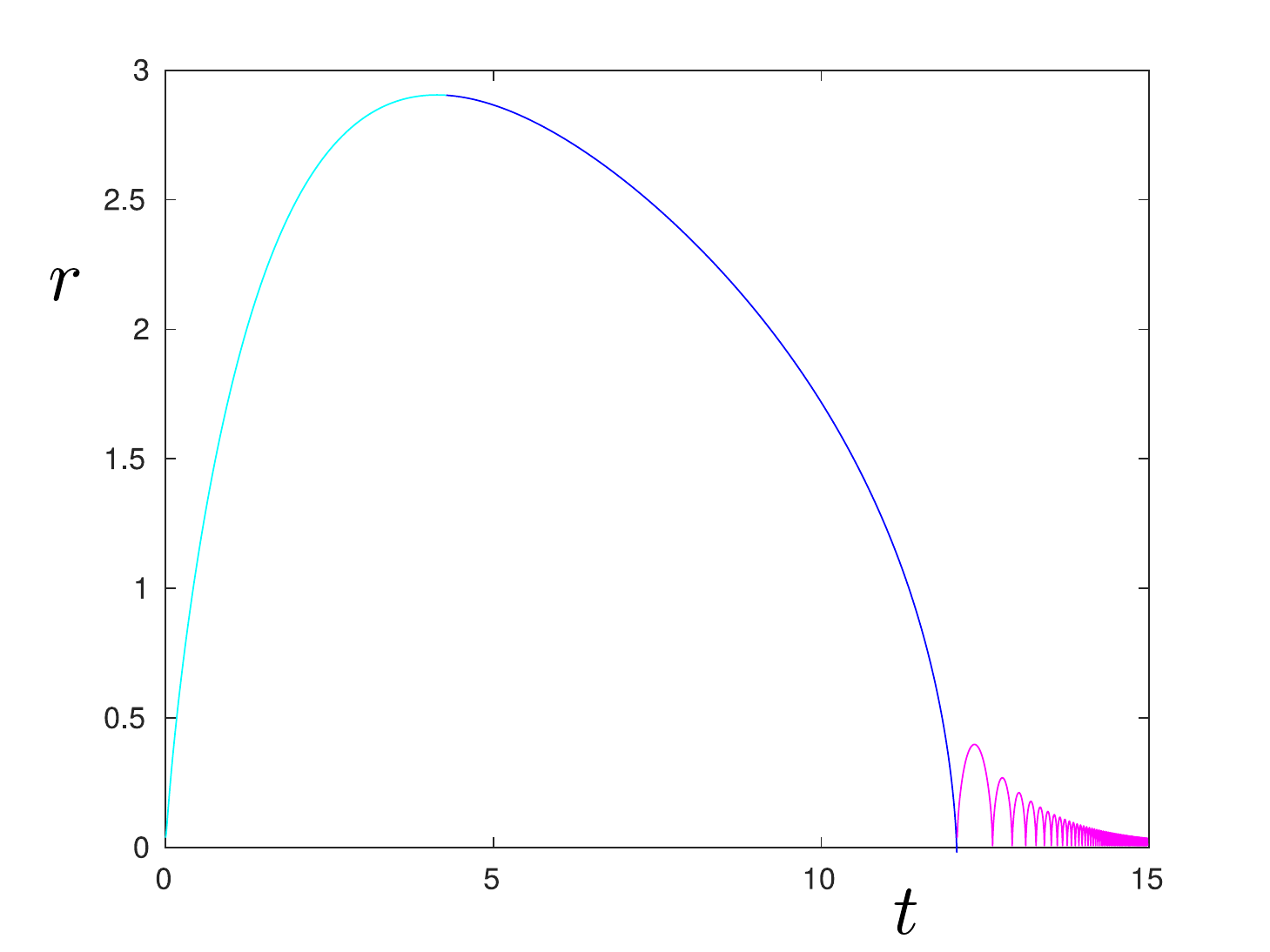}}
 		\caption{$r(t)$ along one of the cyan curves in \figref{simulation}. The curve has been divided into three parts: cyan, blue and purple and these are characterized as follows. The first cyan part corresponds to an increase in $r$ due to the line of saddle points along the $v$-axis in \figref{simulation}(a). This is followed by the blue part, which corresponds to a capture-collision type motion (in phase space this part occurs due to the attraction towards the center manifold $W^c(p_{21}^+)$). The final part in purple, corresponds to repeated ejection-collision type oscillations, the amplitude of which are decreasing. }\figlab{Rvst}
 	\end{center}
 \end{figure}
\bibliography{refs}
\bibliographystyle{plain}
\newpage
\appendix
\section{Proof of \lemmaref{est0}}\applab{est}
We consider \eqref{HNphixeqns}, repeated here for convinience
\begin{equation}\eqlab{HNphixeqns2}
\begin{aligned}
 \frac{dH}{d\tau} &= \frac{1}{3\delta} x\left( \Lambda(H,x)+x^{N} R(H,\phi,x)\right),\\
 \frac{d\phi}{d\tau} & = \frac{1}{3\delta x^2} \left(\Omega(H,x)+x^{N+1} P(H,\phi,x)\right),\\
 \frac{dx}{d\tau}&=-1,
\end{aligned}
\end{equation}
for $x>0$. We then apply $\frac{\partial^{\vert \mathbb \nu\vert}}{\partial \tau^{\nu_1}\partial H_0^{\mathbb \nu_2} \partial \phi_0^{\mathbb \nu_3} \partial x_0^{\partial \nu_4} }$ on both sides of these equations; recall the notation defined in \eqref{znu}. We then obtain $\frac{d}{d\tau}\underline H_{\mathbb \nu}$, $\frac{d}{d\tau}\underline \phi_{\mathbb \nu}$ and $\frac{d}{d\tau}\underline x_{\mathbf \nu}$ on the left hand sides, respectively.  Obviously, 
\begin{align}\eqlab{xnu}
x_{\mathbf \nu} \equiv 0 \, \,\text{unless}\, \, \mathbf \nu=(1,0,0,0),\,\, \text{or}\,\, \mathbf \nu =(0,0,0,1),
\end{align}
where $x_{(1,0,0,0)}=-1,\,x_{(0,0,0,1)}=1$.

To handle the associated right hand sides, we use the Faa di Bruno formula \cite{constantine1996a}: For $\mathbf \nu=(\nu_1,\cdots, \nu_d)\in \mathbb N_0$, $z=(z_1,\cdots,z_d)$, $d\in \mathbb N$, define 
\begin{align*}
 \mathbf \nu ! &:= \prod_{i=1}^d \nu_i !,\\
 D^{\mathbf \nu}&:=\frac{\partial^{\vert \mathbb \nu\vert}}{\partial z_1^{\nu_1}\cdots \partial z_d^{\mathbb \nu_d} },\\
 [z]^{\mathbf \nu}&:=\prod_{i=1}^d z_i^{\nu_i}.
\end{align*}
Moreover, if we also consider $\mathbf \mu=(\mu_1,\ldots,\nu_d)$, then we write $\mathbf \mu\prec \mathbf \nu$ provided at least one of the conditions hold:
\begin{enumerate}
 \item $\vert \mu\vert <\vert \nu\vert$;
 \item $\vert \mu\vert =\vert \nu\vert$ and $\mu_1<\nu_1$; or
 \item $\vert \mu\vert =\vert \nu\vert$, $\mu_1=\nu_1,\ldots,\mu_k=\nu_k$ and $\mu_{k+1}<\nu_{k+1}$ for some $1\le k<d$. 
\end{enumerate}

Finally, we write $\mathbf 0=(0,\ldots,0)$ in $\mathbb N_0^d$. 
\begin{lemma}\lemmalab{faa}
\cite[Theorem 2.1]{constantine1996a}
Consider smooth functions $F:\mathbb R^m\rightarrow \mathbb R$ and $G=(G_1,\ldots,G_m):\mathbb R^d\rightarrow \mathbb R^m$, and define
\begin{align*}
 W(z) := (F\circ G)(z).
\end{align*}
Then 
\begin{align}
D^{\mathbf \nu} W = \sum_{1\le \vert \mathbf \lambda \vert \le n} D^{\mathbf \lambda} F \sum_{s=1}^n \sum_{p_s(\mathbf \nu,\mathbf \lambda)} (\mathbf \nu!) \prod_{j=1}^s \frac{[D^{l_j} G_1,\cdots,D^{l_j} G_m ]^{k_j} }{(k_j!)(l_j)^{\vert k_j\vert}},\eqlab{DnuW}
\end{align}
where $n=\vert \nu\vert$ and 
\begin{equation}\eqlab{ps}
\begin{aligned}
 p_s(\mathbf \nu,\mathbf \lambda):=&\bigg\{(k_1,\ldots,k_s,l_1,\ldots,l_s):\vert k_i\vert >0,\\
 &\quad \mathbf 0\prec l_1\prec\cdots \prec l_s,\\ 
 &\quad \sum_{i=1}^s k_i = \mathbf \lambda,\quad \sum_{i=1}^s \vert k_i\vert l_i = \mathbf \nu\bigg\}.
\end{aligned}
\end{equation}

\end{lemma}
We now consider the following functions:
\begin{align*}
W_R&=\underline x^{N+1}. R\circ (\underline H,\underline \phi,\underline x),\\
W_P& =\underline x^{N-1}. P\circ (\underline H,\underline \phi,\underline x),
\end{align*}
of $(\tau,H_0,\phi_0,x_0)\in V(\xi)$, 
which appear on the right hand side of \eqref{HNphixeqns2}. Here $V(\xi)$ is defined in \eqref{Vxi}, repeated here for convinience:
\begin{align*}
V(\xi):=\left\{(\tau,H_0,\phi_0,x_0) \in (0,\xi)\times J\times \mathbb T\times  (0,\xi)\,:\,0<\tau<x_0\right\}.
\end{align*}
\begin{lemma}\lemmalab{est}
Let $M\in \mathbb N$ be so that $N-2M\ge 0$.
 Suppose that 
 \begin{align*}
  \vert \underline H_{\mathbf \nu}(\tau,H_0,\phi_0,x_0)\vert \le C_M,\quad \vert \underline \phi_{\mathbf \nu}(\tau,H_0,\phi_0,x_0)\vert \vert \tau-x_0\vert^{1+\nu_1+\nu_4} \le C_M,
 \end{align*}
for all $(\tau,H_0,\phi_0,x_0)\in V(\xi)$ and all $1\le \vert \mathbf \nu\vert\le M$. Then there exists a constant $K=K(M,\vert R\vert_{C^M},\vert P\vert_{C^M},C_M)>0$, depending only on: (a) $M$, (b) uniform $C^M$ bounds $\vert R\vert_{C^M}$ and $\vert P\vert_{C^M}$ of $R$ and $P$, respectively, and (c) $C_M>0$, such that
\begin{align*}
 \vert D^{\mathbf \nu} W_R(\tau,H_0,\phi_0,x_0)\vert& \le K\vert \tau-x_0\vert^{N+1-\nu_1-\nu_4-\vert \nu\vert},\\
  \vert D^{\mathbf \nu} W_P(\tau,H_0,\phi_0,x_0)\vert\vert \tau-x_0\vert^2 &\le K\vert t- x_0\vert^{N+1-\nu_1-\nu_4-\vert \nu\vert},
\end{align*}
for all $(\tau,H_0,\phi_0,x_0) \in V(\xi)$ and all $1\le \vert \mathbf \nu\vert\le M$.
\end{lemma}
\begin{proof}
We focus on $W_R$. The estimate of $W_P$ can be obtained in a completely analogous way.
First, we write 
 \begin{align*}
  D^{\mathbf \nu} W_R &= D^{(\nu_1,0,0,\nu_4)} \left(\underline x^{N+1} .D^{(0,\nu_2,\nu_3,0)}\underline R\right),
 \end{align*}
using \eqref{xnu}. Here we use $D^{\nu } \underline R$ to denote the partial differentiation of the composition function  $R(\underline H(),\underline \phi(),\underline x())$. Then by the product rule, we have
\begin{equation}\eqlab{DW}
\begin{aligned}
 D^{\mathbf \nu} W_R = \sum_{q_4=0}^{\nu_4}\sum_{q_1=0}^{\nu_1} (-1)^{\nu_1-q_1}& \begin{pmatrix}
                                                                \nu_4\\
                                                                q_4
                                                               \end{pmatrix} \begin{pmatrix}
                                                                \nu_1\\
                                                                q_1
                                                               \end{pmatrix}
                                                               \begin{pmatrix}
                                                                N+1\\
                                                                \nu_4-q_4+\nu_1-q_1
                                                               \end{pmatrix}
                                                               (\nu_4-q_4+\nu_1-q_1)! \\
                                                               &\times \underline x^{N+1-\nu_4+q_4-\nu_1+q_1} D^{(q_1,\nu_2,\nu_3,q_4)}\underline R,
                                                               \end{aligned}
                                                               \end{equation}
 using \eqref{xnu} again. We then use \lemmaref{faa}, in particular \eqref{DnuW} with $\nu=(q_1,\nu_2,\nu_3,q_4)$, $d=4$, $m=3$, $D^{l_{j}} G_1=\underline H_{l_{j}}$, $D^{l_{j}} G_2=\underline \phi_{l_{j}}$, $D^{l_{j}} G_3=\underline x_{l_{j}}$, to estimate $D^{(q_1,\nu_2,\nu_3,q_4)} \underline R$:
 \begin{align*}
  \vert D^{(q_1,\nu_2,\nu_3,q_4)} \underline R\vert &\le K_0 \sum_{p_s((q_1,\nu_2,\nu_3,q_4),\mathbf \lambda)} \prod_{j=1}^s \vert \underline H_{l_{j}}\vert^{k_{j,1}} \vert \underline \phi_{l_{j}}\vert^{k_{j,2}}\vert \underline x_{l_j}\vert^{k_{j,3}}\\
  &\le K_0  \sum_{p_s((q_1,\nu_2,\nu_3,q_4),\mathbf \lambda)} \prod_{j=1}^s C_M^{\vert k_{j}\vert} \vert \underline x\vert^{-k_{j,2}(1+l_{j,1}+l_{j,4})}\\
  &\le K_0 C_M^{M} \vert \underline x\vert^{-q_1-q_4},
 \end{align*}
 writing $k_j=(k_{j,1},k_{j,2},k_{j,3})$, $l_j = (l_{j,1},\ldots,l_{j,4})$,
 for some $K_0=K_0(M,\vert R\vert_{C^M})>0$ depending only on: (a) $M$ and (b) the $C^M$ bound $\vert R\vert_{C^M}$ on the smooth function $R$. In the final inequality, we have used the definition of $p_s$ \eqref{ps}. Specifically, from
\begin{align}
 \sum_{j=1} k_{j,2} (1+l_{j,1}+l_{j,2})<\sum_{j=1} \vert k_{j}\vert (1+l_{j,1}+l_{j,2})\le \sum_{j=1} \vert \nu\vert  (1+l_{j,1}+l_{j,2})\le 1+q_1+q_2.\eqlab{ineq}
\end{align}
 we have concluded that $\sum_{j=1} k_{j,2} (1+l_{j,1}+l_{j,2})\le q_1+q_2$; notice that the first inequality in \eqref{ineq} is \textit{strict}. 
Now using \eqref{DW}, we have
 \begin{align*}
  \vert D^{\mathbf \nu} W_R \vert \le K \vert \tau-x_0\vert^{N+1-\nu_1-\nu_4},
 \end{align*}
for $K=K(M,\vert R\vert_{C^M},C_M)>0$ large enough. 
 \end{proof}
%

Finally, we describe the
the functions
\begin{align*}
 W_\Lambda=\underline x .\Lambda \circ (\underline H,\underline x),\quad W_\Omega =\underline x^{-2}. \Omega\circ (\underline H,\underline x),
\end{align*}
that also appear on the right hand side \eqref{HNphixeqns2}, but are independent of $\underline \phi$.
\begin{lemma}\lemmalab{est2}
 Let $M\in \mathbb N$ and suppose that
 \begin{align*}
  \vert \underline H_{\mathbf \nu}(\tau,H_0,\phi_0,x_0)\vert \le C_M,
 \end{align*}
for all $(\tau,H_0,\phi_0,x_0)\in V(\xi)$ and all $1\le \vert \mathbf \nu\vert\le M$. 
Then there exists a constant $K=K(M,\vert D\vert_{C^M},\vert \Omega\vert_{C^M},C_M)>0$, depending only on: (a) $M$, (b) uniform $C^M$ bounds $\vert D\vert_{C^M}$ and $\vert \Omega \vert_{C^M}$ of $D$ and $\Omega$, respectively, and (c) on a constant $C_M>0$, such that
\begin{align*}
\vert D^{\mathbf \nu}W_\Lambda(\tau,H_0,\phi_0,x_0)\vert\le K,\quad  \vert D^{\mathbf \nu}W_\Omega(\tau,H_0,\phi_0,x_0)\vert\vert \tau-x_0\vert^{2+\nu_1+\nu_4} \le K,
\end{align*}
for all $(\tau,H_0,\phi_0,x_0)\in V(\xi)$ and all $1\le \vert \mathbf \nu\vert\le M$.
\end{lemma}
\begin{proof}
 Follows from a direct calculation.
\end{proof}

\begin{lemma}\lemmalab{final}
Let $M=\lfloor \frac{N}{2}\rfloor$. Then for any $\mathbf \nu\in \mathbb N_0$ with $\vert \nu\vert \le M$ there exists a constant $C_{\mathbf \nu}$ and a $\xi_{\mathbb \nu}>0$ such that 
\begin{align}
 \vert \underline H_{\mathbf\nu}(\tau,H_0,\phi_0,x_0)\vert \le C_{\mathbf \nu},\quad \vert \underline \phi_{\mathbf\nu}(\tau,H_0,\phi_0,x_0)\vert\vert \tau-x_0\vert^{1+\nu_1+\nu_4}\le C_{\mathbf \nu},\eqlab{estfinal0}
\end{align}
for all $(\tau,H_0,\phi_0,x_0)\in V(\xi_{\mathbf \nu})$. 
\end{lemma}
Upon proceeding as in the proof of \lemmaref{contH0}, \eqref{estfinal0} implies that $\underline H_{\mathbf \nu}$ extends continuously to $\overline{V(\xi)}$. Consequently, we complete the proof of \lemmaref{est0} by proving \lemmaref{final}.

To prove \lemmaref{final}, we proceed by induction. It is true for $\mathbf \nu=\textbf 0$, see \eqref{Hphibound} and the proof of \lemmaref{H0phi0}. Next, suppose that it is true for $\vert \mathbf \nu\vert =n$. We then consider $\mathbf \nu':=\mathbf \nu+\mathbf \delta_i$, with $\mathbf \delta_i = (\delta_{i,1},\ldots,\delta_{i,4})$ for
\begin{align*}
 \delta_{i,j} = \begin{cases}
                 1 & j=i,\\
                 0 & j\ne i.
                \end{cases}
\end{align*}
\begin{lemma}\lemmalab{ic}
 For $\widetilde C_{\mathbf \nu'}>0$ large enough and $\xi>0$ small small enough, we have that 
\begin{align*}
 \vert \underline H_{\mathbf\nu'}(0,H_0,\phi_0,x_0)\vert\le  \widetilde  C_{\mathbf \nu'},\quad  \vert \underline \phi_{\mathbf\nu'}(0,H_0,\phi_0,x_0)\vert \vert x_0\vert^{1+\nu_1+\nu_4+\delta_{i,1}+\delta_{i,4}} \le \widetilde  C_{\mathbf \nu'},
\end{align*}
for all $(0,H_0,\phi_0,x_0)\in V(\xi)$
\end{lemma}
\begin{proof}
 By definition $\underline H(0,H_0,\phi_0,x_0)=H_0$, $\underline \phi(0,H_0,\phi_0,x_0)=\phi_0$, $\underline x(0,H_0,\phi_0,x_0)=x_0$. Consequently, if $\nu=(0,\nu_2,\nu_3,\nu_4)$ then $\underline z_{\mathbf \nu}(0,\ldots)=0$ for $z=H,\phi,x$ (unless $\underline H_{(0,1,0,0)}=1$, $\underline \phi_{(0,0,1,0)}=1$, or $\underline x_{(0,0,0,1)}=1$). On the other hand, if $\nu_1\ge 1$ then 
by \eqref{HNphixeqns2}
 \begin{align*}
  \underline H_{\mathbb \nu+\mathbb \delta_{i}}\big\vert_{\tau=0} &= \frac{1}{3\delta}\left( D^{(\nu_1-1,\nu_2,\nu_3,\nu_4)+\mathbb \delta_{i}}(\underline x. \Lambda(\underline H,\underline x))_{\tau=0}+D^{(\nu_1-1,\nu_2,\nu_3,\nu_4)+\mathbb \delta_{i}} (\underline x^{N+1}. R(\underline H,\underline \phi,\underline  x))_{\tau=0}\right),\\
  \underline \phi_{\nu+\mathbb \delta_{i}}\big\vert_{\tau=0} &=\frac{1}{3\delta}\left( D^{(\nu_1-1,\nu_2,\nu_3,\nu_4)+\mathbb \delta_{i}}(\underline x^{-2} .\Omega(\underline H,\underline x))_{\tau=0}+D^{(\nu_1-1,\nu_2,\nu_3,\nu_4)+\mathbb \delta_{i}} (\underline x^{N-1}. P(\underline H,\underline \phi,\underline  x))\right)_{\tau=0}.
 \end{align*}
The result then follows upon using the induction hypothesis, \lemmaref{est} and \lemmaref{est2}.

 \end{proof}
We now consider 
\begin{align*}
 \frac{d \underline H_{\mathbf \nu'}}{d\tau} &=\frac{1}{3\delta}\left( D^{\mathbf \nu'} (\underline x. \Lambda(\underline H,\underline x))+D^{\mathbf \nu'} (\underline x^{N+1} .R(\underline H,\underline \phi,\underline  x))\right),\\
 \frac{d \underline \phi_{\mathbf \nu'}}{d\tau} &=\frac{1}{3\delta}\left( D^{\mathbf \nu'} (\underline x^{-2} .\Omega(\underline H,\underline x))+D^{\mathbf \nu'} (\underline x^{N-1} .P(\underline H,\underline \phi,\underline  x))\right).
\end{align*}
Let $C_{H,\mathbf \nu'}:=2\widetilde  C_{\mathbf \nu'}$ and $C_{\phi,\mathbf \nu'}>\widetilde  C_{\mathbf \nu'}$. We may take $\widetilde  C_{\mathbf \nu'}$ large enough such that $\widetilde  C_{\mathbf \nu'}\ge C_{\nu}$ for all $\vert \nu\vert\le n$. Then by \lemmaref{ic}, we have that
 \begin{align}
 \vert \underline H_{\mathbf\nu'}(\tau,H_0,\phi_0,x_0)\vert \le C_{H,\mathbf \nu'},\quad \vert \underline \phi_{\mathbf\nu'}(\tau,H_0,\phi_0,x_0)\vert\vert \tau-x_0\vert^{1+\nu_1+\nu_4+\delta_{i,1}+\delta_{i,4}}\le C_{\phi,\nu'},\eqlab{estfinal}
\end{align}
for all $\tau\in (0,\tau_0)$, with $\tau_0(H_0,\phi_0,x_0)>0$ small enough. In fact, due to \lemmaref{est}, and the induction hypothesis, we obtain that
\begin{align*}
D^{\mathbf \nu'} (\underline x^{N+1} .R(\underline H,\underline \phi,\underline  x))\rightarrow 0,\quad
D^{\mathbf \nu'} (\underline x^{N-1} .P(\underline H,\underline \phi,\underline  x))\rightarrow 0
\end{align*}
as $\tau,x_0\rightarrow 0$ for any $C_{H,\mathbf \nu'}, C_{\phi,\mathbf \nu'}>0$. Consequently, by \lemmaref{est2}, we find for $\xi>0$ small enough, that there is a constant $K=K(M,\vert \Lambda\vert_{C^M},\vert \Omega\vert_{C^M}, \widetilde C_{\mathbf \nu'})$ depending only on: (a) $M$, (b) uniform $C^M$ bounds $\vert \Lambda\vert_{C^M}$ and $\vert \Omega\vert_{C^M}$ of $\Lambda$ and $\Omega$, respectively, and on (c) $\widetilde C_{\mathbf \nu'}$, such that 
\begin{align*}
\big \vert \frac{d \underline H_{\mathbf \nu'}}{d\tau}\big \vert &\le K,\\
 \big \vert \frac{d \underline \phi_{\mathbf \nu'}}{d\tau}\big\vert &\le \underline x^{-2-\nu_1-\nu_4-\delta_{i,1}-\delta_{i,4}} K,
\end{align*}
for all $(\tau,H_0,\phi_0,x_0)\in V(\xi)$ and $\tau\in (0,\tau_0)$ . The main observation here is that by taking $\xi>0$ small enough, we ensure that $K$ is independent of $C_{H,\nu'}$ and $C_{\phi,\mathbf \nu'}$. We then integrate and use 
$$ \bigg\vert \int_{0}^\tau \underline x^{-q} dt\bigg \vert \le \frac{2}{q-1} \vert \tau-x_0\vert^{1-q},$$  for any $q>1$. This produces the following
\begin{align*}
 \vert \underline H_{\nu'}(\tau,H_0,\phi_0,x_0)\vert &\le \widetilde  C_{\mathbf \nu'}+\vert x_0\vert K,\\
 \vert \underline \phi_{\nu'}(\tau,H_0,\phi_0,x_0)\vert  &\le \widetilde  C_{\mathbf \nu'}\vert x_0\vert^{-1-\nu_1-\nu_4-\delta_{i,1}-\delta_{i,4}} + \frac{2}{1+\nu_1+\nu_4+\delta_{i,1}+\delta_{i,4}} \vert \tau-x_0\vert^{-1-\nu_1-\nu_4-\delta_{i,1}-\delta_{i,4}} K\\
 &\le \vert \tau-x_0\vert^{-1-\nu_1-\nu_4-\delta_{i,1}-\delta_{i,4}} \left(\widetilde  C_{\mathbf \nu'}+\frac{2K}{1+\nu_1+\nu_4+\delta_{i,1}+\delta_{i,4}} \right),
\end{align*}
where we have also used \lemmaref{ic}. We now take 
\begin{align*}
 C_{\phi,\mathbf \nu'}:=\widetilde  C_{\mathbf \nu'}+\frac{2K}{1+\nu_1+\nu_4+\delta_{i,1}+\delta_{i,4}},
\end{align*}
and upon decreasing $\xi>0$ further, we ensure that $\vert x_0\vert \le \frac{\widetilde  C_{\mathbf \nu'}}{ K}$.  
Then it follows that \eqref{estfinal} holds true for all $\tau\in (0,x_0)$
as desired. This completes the proof of \lemmaref{final}.



\end{document}